\documentclass[10 pt]{amsart}

\usepackage{geometry}
\usepackage[colorlinks=true,linkcolor=blue,urlcolor=blue,
  citecolor=blue]{hyperref}

\usepackage{amsfonts, amsmath, amssymb, mathrsfs}
\usepackage{amsthm, 
url}
\usepackage[all]{xy}
\usepackage{graphicx, 
epstopdf}
\usepackage{subfig}
\usepackage{color}
\usepackage{bbm}
\usepackage{subfloat}
\usepackage[draft]{fixme}
\usepackage{bm}
\usepackage{epigraph}

\usepackage{ifpdf}

\ifpdf
\else
\fi

\usepackage[numbers, sort]{natbib}

\makeatletter
\@namedef{subjclassname@2010}{%
 \textup{2010} Mathematics Subject Classification}
\makeatother

\makeatletter
\newif\if@restonecol
\makeatother

\usepackage[boxed,vlined,lined]{algorithm2e}
\usepackage{algorithmic}

\newtheorem{theorem}{Theorem}[section]
  \newtheorem{lemma}[theorem]{Lemma}

   \newtheorem{corollary}[theorem]{Corollary}
  \newtheorem{proposition}[theorem]{Proposition}

 \newtheorem{definition}[theorem]{Definition}
   \newtheorem{remark}[theorem]{Remark}
\theoremstyle{definition}
\newtheorem{exampleth}[theorem]{Example}
\newenvironment{example}{\begin{exampleth}}{\hfill $\diamond$\\ \end{exampleth}}

\newcommand \NN {\mathbb{N}} 
\newcommand \QQ {\mathbb{Q}} 
\newcommand \ZZ {\mathbb{Z}}
\newcommand \CC {\mathbb{C}}
\newcommand \RR {\mathbb{R}}

\newcommand \pr {\mathbb{P}}

\newcommand \codim{\textrm{codim}}

\newcommand \Sec {\textrm{Sec}}

\newcommand \Sn {\mathbb{S}}
\newcommand \init {\textrm{in}}
\newcommand \NP {\textrm{NP}}

\newcommand \cT {\mathcal{T}}

\newcommand \cP {\mathcal{P}}

\newcommand \ww {\omega}
\newcommand \TP {\mathbb{T}}
\newcommand \gfan {\texttt{gfan}}
\newcommand \aaa {\underline{a}}
\newcommand \ap {\underline{a}'}
\newcommand{\val}{\text {val}}
\newcommand{\refl}{\text{refl}}

\author{Mar\'ia Ang\'elica Cueto}

\address{Mathematics Department,
  Columbia University, MC 4403, New York, NY, 10027, USA}
\email{macueto@math.columbia.edu}

\author{Shaowei Lin}
 \address{Department of Mathematics,
  University of California, Berkeley, CA 94720, USA}
\email{shaowei@math.berkeley.edu}

\thanks{M.A.\ Cueto was
    supported by a UC Berkeley Chancellor's Fellowship and S.\ Lin was supported by a Singapore A*STAR Fellowship.}

\title{Tropical secant graphs of monomial curves}


\begin{document}

\begin{abstract}
  The first secant variety of a projective monomial curve is a
  threefold with an action by a one-dimensional torus. Its
  tropicalization is a three-dimensional fan with a one-dimensional
  lineality space, so the tropical threefold is represented by a
  balanced graph. Our main result is an explicit construction of that
  graph.  As a consequence, we obtain algorithms to effectively compute
  the multidegree and Chow polytope of an arbitrary projective
  monomial curve. This generalizes an earlier degree formula due to
  Ranestad. The combinatorics underlying our construction is rather
  delicate, and it is based on a refinement of the theory of geometric
  tropicalization due to Hacking, Keel and Tevelev.
\end{abstract}

\maketitle 

\section{Introduction}
\label{sec:introduction}

In this paper, we define and study four graphs that hide rich
geometry: an abstract graph (the \emph{abstract tropical secant
  surface graph}), a weighted graph in $\RR^{n+1}$ (the \emph{tropical
  secant surface graph} or \emph{master graph}), a weighted graph in
the $(n-2)$-sphere (the \emph{tropical secant graph}) and, finally, a
weighted graph representing a simplicial spherical complex 
(the \emph{Gr\"obner tropical secant graph}). 
All four graphs are parameterized by a sequence of $n$ coprime
distinct positive integers $i_1, \ldots, i_n$, where $n\geq 4$. As
their names suggest, these graphs are stepping stones to constructing
either a tropical surface or the tropicalization of a secant
variety.

In recent years, tropical geometry has provided a new approach to
attack implicitization problems \citep{TropDiscr, Mega09, ElimTheory,
  NPofImplicitEquation}. 
We tropicalize classical varieties
to obtain weighted polyhedral fans with the hope of recovering useful
algebro-geometric information by working on the polyhedral-geometric
side. Our paper illustrates this principle with a family of classical
secant threefolds: the first secant variety of a monomial projective curve
$(1: t^{i_1}: \ldots: t^{i_n})$.  By definition, the secant variety of
the curve is the closure of the union of all lines that meet the curve
in two distinct points. These varieties have been studied extensively
in the literature; see \citep{CoxSidman07, SecantMonCurves} and
references therein for more details. One of the main contributions of
this paper is a complete characterization of their tropical
counterparts, which is carried out in full detail in
 Section~\ref{sec:trop-secant-graph}. More precisely,
\begin{theorem}\label{thm:MainThm}
  Given a monomial curve $C$ in $ \pr^n$ parameterized by $n$ distinct
  coprime positive integers $\{i_1, \ldots, i_n\}$, the
  tropicalization of its first secant variety is the cone from the
  vector space $\RR\langle \mathbf{1}, (0, i_1, \ldots,
  i_n)\rangle$ over the tropical secant graph of the curve.
\end{theorem}

Strictly contained in this
tropical variety is the first tropical secant complex of the
monomial curve (Propositions~\ref{pr:TropSecantComplex}
and~\ref{pr:tropSecantVsTropicalizedSecant}). This complex has
recently been investigated by Develin and Draisma in~\cite{DevelinTropSecant,
  DraismaTropSecant} in an attempt to study 
secant
varieties of 
toric varieties via tropicalizations.


Unfortunately, computing the tropicalization of an algebraic variety
without any information about its defining ideal is not an easy
task. This new point of view was pioneered by the work of Kapranov and
his collaborators~\cite{EKL}, and further developed by Hacking, Keel
and Tevelev~\cite{HKT}, and by the first
author~\cite{Implicitizationsurfaces}.  This new theory, known as
\emph{geometric tropicalization} (Theorem~\ref{thm:GT}), relies
on a parametric representation of the variety and the characterization
of tropicalizations of algebraic varieties in terms of
\emph{divisorial valuations}, following the spirit of~\citep{BG}. To
do so, we need to provide a normal $\QQ$-factorial compactification of
the given variety, satisfying suitable boundary properties.  This can
be quite difficult to perform if the variety is non-generic. We can
see this from the extensive number of pages we devote to computing the
master graph using this technique, and also from the small sample of
numerical examples
available in the
literature.

As we explain in Section~\ref{sec:comb-monom-curv}, 
the main obstacle to
apply this theory for non-generic surfaces lies in finding a suitable
(tropical) compactification of the given variety whose boundary has
simple normal crossings
, a condition which can be further relaxed to \emph{combinatorial
  normal crossings}~\citep{Implicitizationsurfaces,
  ElimTheory}. In the surface case,  This last condition requires a 
 divisorial boundary
 where no three boundary components meet at a point.
In principal, this can be achieved by modifying any given
compactification by blow-ups of isolated surface
singularities, and the difficulty becomes algebraic, since we need to
carry all valuations along the different blow-ups.



In practice, knowing which points to blow up
and how to carry the geometric information on the boundary
along 
the
various blow-ups performed can be a combinatorial challenge.  However,
the surfaces studied in this paper 
have a very rich combinatorial structure, and we can make full use of
this feature to compute their tropicalizations using resolutions. Our
methods allow us to read off this tropical variety 
directly from the master graph, which encodes the resolution
diagrams of the surface at each singular point (see
Figure~\ref{fig:OriginAndInfinityAndRest}). This is explained in
detail in 
Sections~\ref{sec:tropical-geometry} and~\ref{sec:comb-monom-curv}, in
particular in Theorem~\ref{thm:MasterGraphIsTropical}. 
This construction provides a compactification of the toric arrangement
given by the $n+1$ binomial curves $(w^{i_j}-\lambda=0)$ in $\TP^2$,
for $0\leq j\leq n$.  Such compactifications have been studied
recently by L.\ Moci~\cite{MociPaper}. His construction, closely
related to ours in spirit, realizes the wonderful compactification of
De Concini and Procesi~\cite{Wonderful}.



In Section~\ref{sec:newt-polyt-secant}, we exploit these tropical
secant graphs to recover geometric information about the first secant
varieties of monomial curves. More precisely, we recover their
multidegrees with respect to the rank two lattice generated by the
all-one's vector and the exponent vectors of the curves, following
algorithms from~\cite{TropDiscr, Mega09}. The degree of these
threefolds was previously worked out by Ranestad
in~\citep{SecantMonCurves} and, unsurprisingly, our methods give
similar combinatorial formulas 
in terms of the
exponents $i_1, \ldots, i_n$. The main advantage of our approach is
that, with the same effort, we can provide much more information about
theses varieties, including their \emph{Chow polytopes}.

Our construction is particularly enlightening in the case when $n=4$,
where the threefold secant variety becomes a hypersurface. In this
special situation, we recover the Newton polytope of the defining
equation. Although the lack of a fan structure in our description of
this tropical variety is not an issue in our methods, it would be
desirable to have one to predict extra combinatorial information about
the Newton polytope, such as the number of facets. For this reason, we
devote the last part of Section~\ref{sec:newt-polyt-secant} to
refining the presentation of the tropical secant graphs to turn them into
weighted simplicial complexes. These structures are inherited from the
Gr\"obner fan structure of the defining ideals and
from the coarsest fan structure of the tropical threefolds. We choose
the name \emph{Gr\"obner tropical secant graph} to highlight this
property.  We illustrate all our constructions and results with the
sequence $\{30,45,55,78\}$, inspired by \cite[Example
3.3]{SecantMonCurves}, and with the rational normal curve
(Example~\ref{ex:RNC}).

Although secant varieties have been extensively studied in the past,
we hope our work illustrates the power of \emph{tropical implicitization}
and how it can be used to go beyond standard implicitization methods even when
looking at classical examples. 

\section{The master graph}
\label{sec:fundamental-graph}

In this section, we describe the main object of this paper: the master
graph. We start by defining an abstract graph parameterized by $n$
coprime positive integers $i_1, \ldots, i_n$.  Throughout the paper,
we set $n\geq 4$. To simplify notation, we call $i_0=0$ and we assume
$0<i_1<i_2<\ldots<i_n$.  We build this graph by gluing two types of
graphs along common labeled nodes: two caterpillar graphs $G_{E,D}$,
$G_{h,D}$ and a family of star graphs $\{G_{F_{\aaa},D}\}_{\aaa}$
parameterized by suitable subsets $\aaa$ of the index set $\{0, i_1,
\ldots, i_n\}$. 
We call this abstract graph
the \emph{abstract tropical secant surface graph}.

Our first building block is the caterpillar graph $G_{E,D}$, as
illustrated on the top-left side of
Figure~\ref{fig:OriginAndInfinityAndRest}. It consists of $2n-1$ nodes
and $2n-2$ edges. Nodes are grouped in two levels, with labels
$E_{i_1}, \ldots, E_{i_{n-1}}$ and $D_{i_1}, \ldots,
D_{i_n}$. Similarly, our second graph, denoted by $G_{h,D}$ and
depicted on the bottom-left side of
Figure~\ref{fig:OriginAndInfinityAndRest}, consists of $2n$ nodes
grouped in two levels with labels $h_{i_1}, \ldots, h_{i_{n-1}}$ and
$D_0, D_{i_1}, \ldots, D_{i_n}$ respectively, and $2n-1$ edges.

The third family of graphs consists of star trees and it is denoted by
$\{G_{F_{\aaa}, D}\}_{\aaa}$. These graphs are parameterized by sets 
  of size at least two, obtained by intersecting an arithmetic
  progression of integer numbers with the index set $\{0, i_1, \ldots,
  i_n\}$. They are
  illustrated in the rightmost picture in
  Figure~\ref{fig:OriginAndInfinityAndRest}. We allow the common
  difference of these progressions to be 1, so the set of all
  exponents $\{0, i_1, \ldots, i_n\}$ is a valid subset. The size of
  the subset $\aaa$ 
  associated to an
  arithmetic progression  coincides with the degree of the
  corresponding node $F_{\aaa}$ in the abstract graph.  Note that
  several arithmetic progressions can give the same subset of $\{0,
  i_1,\ldots, i_n\}$, and hence the same node $F_{\aaa}$ in the graph
  $G_{F_{\aaa}, D}$.  If $\aaa =\{i_{j_1}, \ldots, i_{j_k}\}$, then
  the graph has $k+1$ nodes and $k$ edges: $k$ nodes labeled
  $D_{i_{j_1}}, \ldots, D_{i_{j_k}}$ and a central node $F_{\aaa}$,
  connected to the other $k$ nodes in the graph.  As
  Example~\ref{ex:RanestadMaster} reveals, only nodes of degree at
  least three are relevant for our constructions, so in principle
  we should only consider subsets of size at least three. However, to
  simplify our statements, we allow subsets of
  size two as well.
\begin{figure}[htb]
  \centering
\hspace{-.3cm}
 \begin{minipage}[c]{.665\linewidth}
\hspace{1.05cm} \includegraphics[scale=0.35]{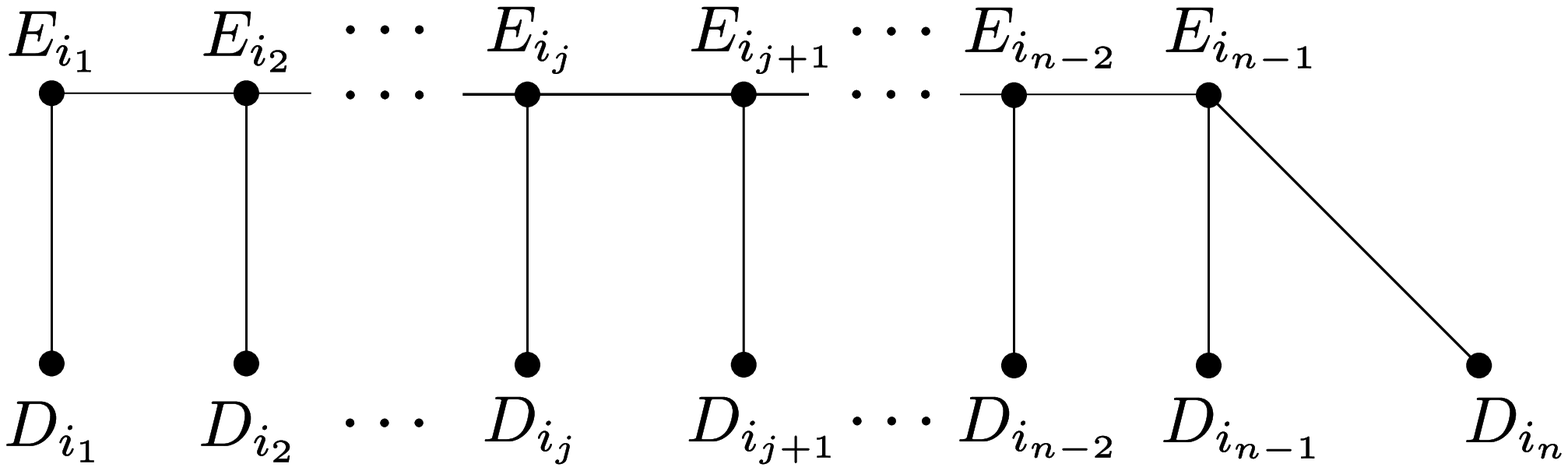}

 \vspace{3ex}

 \includegraphics[scale = 0.35]{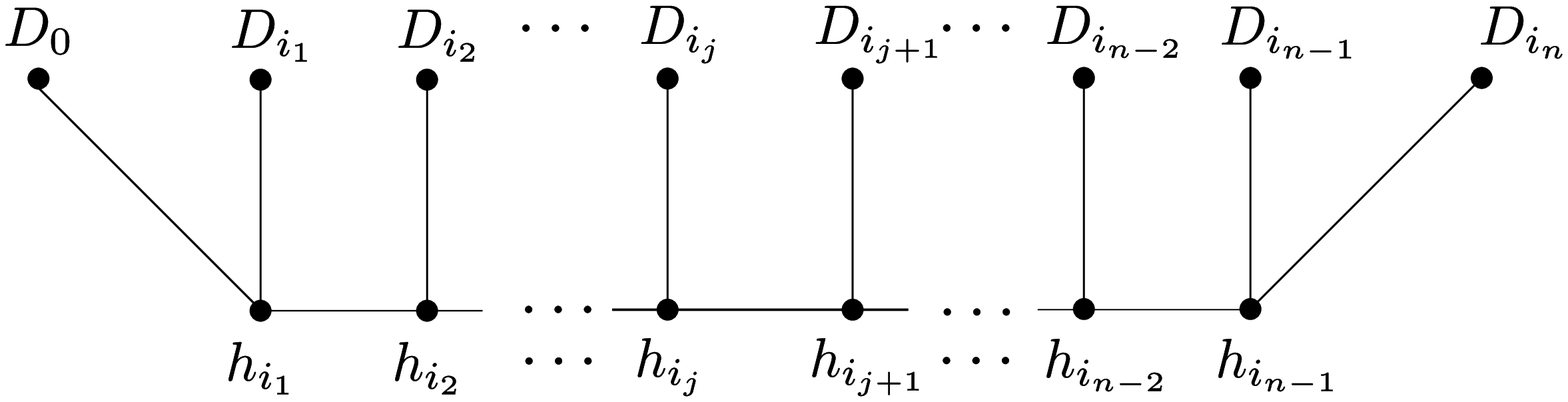}
 \end{minipage}
 \begin{minipage}[c]{.31\linewidth}
\begin{center}\;\;   \includegraphics[scale=0.32]{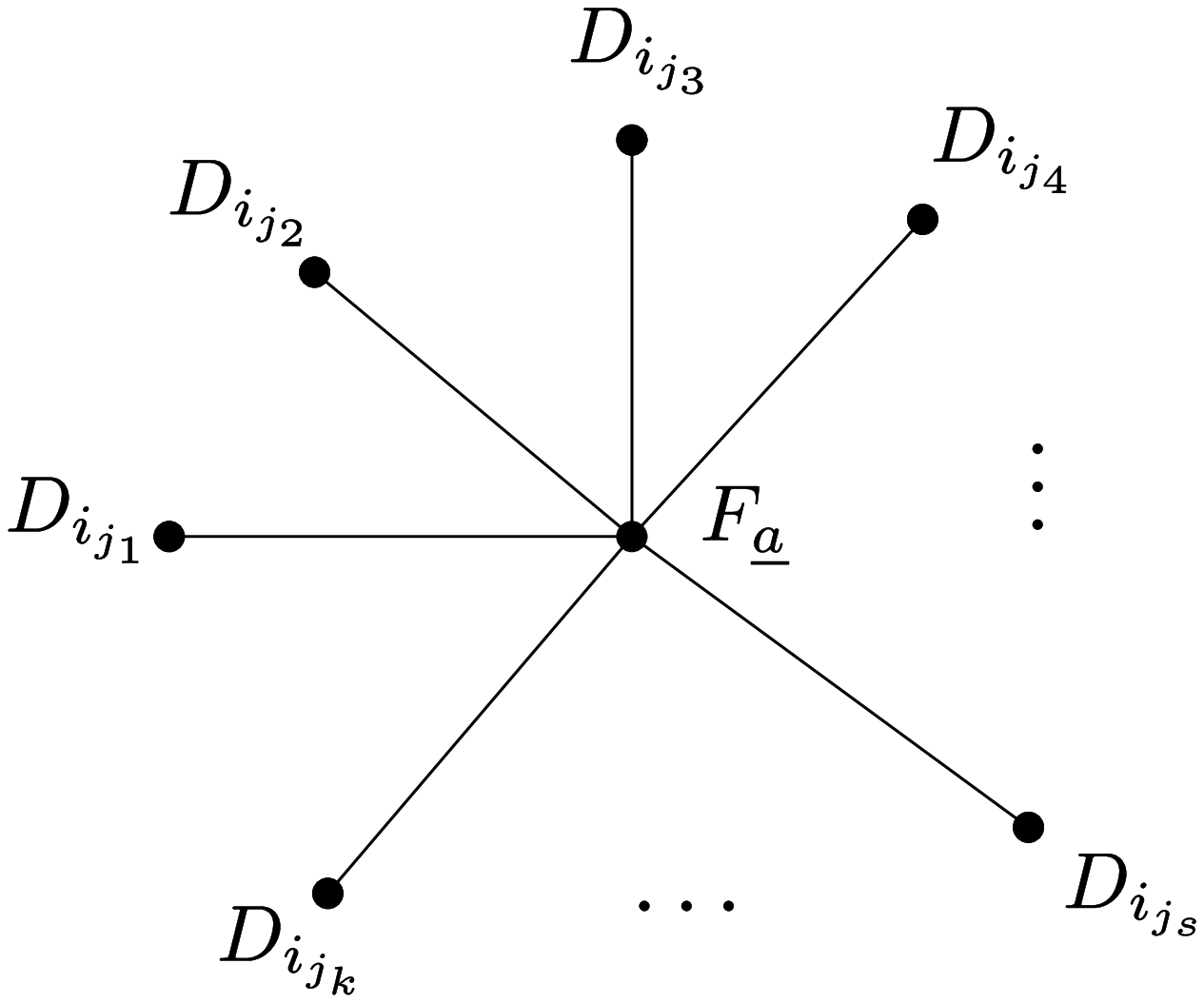}\\
\ \\

$\underline{a}=\{i_{j_1}, \ldots, i_{j_k}\}$
\end{center}
 \end{minipage}
  \caption{The graphs $G_{E,D}$, $G_{h,D}$ and $\{G_{F_{\aaa}, D}\}_{\aaa}$ glue
    together
    to form the \emph{master graph}.}
  \label{fig:OriginAndInfinityAndRest}
\end{figure}

We 
realize the abstract tropical secant surface graph in $\RR^{n+1}$ by
mapping each node to an integer vector and extending linearly on all
edges. Our chosen map has additional data: a weight for each edge in
the graph.  We call this weighted graph the \emph{tropical secant
  surface graph} or \emph{master graph}. We explain this construction
in full detail below. For a numerical example, see
Figure~\ref{fig:RanestadMaster}.

  \begin{definition}\label{def:masterGraph}
    The \emph{master graph} is a weighted graph in $\RR^{n+1}$
    parameterized by $n$ distinct coprime numbers $\{i_1, \ldots,
    i_n\}$ with nodes:
    \begin{enumerate}
    \item $D_{i_j}= {e_j}:=(0, \ldots, 0,
      1, 0, \ldots, 0)$ \qquad ($ 0\leq j\leq n$),
    \item $E_{i_j}=(0, i_1, \ldots, i_{j-1},
      i_j, \ldots, i_j)$ , $h_{i_j}=(-i_j, \ldots,
      -i_j, -i_{j+1}, \ldots, -i_n)$ \quad ($ 1\leq j\leq n-1$),
    \item $F_{\underline{a}}=\sum_{i_j\in \underline{a}}{ e_j}$\quad
      where $\underline{a}\subseteq\{0, i_1, \ldots, i_n\}$ has size
      at least two and is obtained by intersecting an arithmetic
      progression of integers with the index set $\{0, i_1, \ldots, i_n\}$.
    \end{enumerate}
Its edges agree with the edges of the abstract
tropical secant surface graph, and have weights:
\begin{enumerate}
\item $m_{D_{i_0}, h_{i_1}}\!=1$ , $m_{D_{i_n},
    E_{i_{n-1}}}\!\!=gcd(i_1, \ldots, i_{n-1})$ , $m_{D_{i_n},
    h_{i_{n-1}}}\!\!=i_n$,
\item $m_{D_{i_j},E_{i_j}}=\gcd(i_1, \ldots, i_j)$ , $m_{D_{i_j},
    h_{i_j}} = \gcd(i_j, \ldots, i_{n})$ \qquad ($ 1\leq j\leq n-1$),
    \item $m_{E_{i_j}, E_{i_{j+1}}}\!\! =\! \gcd(i_1, \ldots, i_j)$ ,
$m_{h_{i_j}, h_{i_{j+1}}} \!\!= \!\gcd(i_{j+1}, \ldots, i_n)$
     \;\; 
($ 1\leq j \leq n-2$),
\item $m_{F_{\underline{a}}, D_{i_j}}=\sum_r \varphi(r)$, where we sum
  over the common differences $r$ of all arithmetic progressions
  containing $i_j$ and giving the same subset $\underline{a}$.  Here,
  $\varphi$ denotes Euler's phi function.
    \end{enumerate}
  \end{definition}

\begin{remark}
  As we mentioned earlier, if the subset $\aaa$ 
  has two elements, say $i_j$ and $i_k$, then
  $F_{\aaa}$ is a bivalent node and we may eliminate it from the graph
  if desired, replacing its two adjacent edges by a single edge. Both
  edges $F_{i_j,i_k}D_{i_j}$ and $F_{i_j,i_k}D_{i_k}$ have the
  same multiplicity, so we assign this number as the multiplicity of
  the new edge $D_{i_j}D_{i_k}$.

  From the definition, it is immediate to check that the node $E_{i_1}$
  is always bivalent. However, we always keep it in our
  graph, since this greatly simplifies our constructions. 
\end{remark}

We illustrate the previous definition 
with an example.  Note that, in general, the master graph may have
nodes $F_{\aaa}$ with $0\notin
\aaa$. 
This is determined by the combinatorics of the set $\{i_1,
\ldots, i_n\}$.

\begin{example}\label{ex:RanestadMaster}
We compute the master graph associated to the set $\{30,45,55,78\}$. For
simplicity, we eliminate all nine bivalent nodes $F_{i_j,i_k}$ from the
construction and we keep the bivalent grey node $E_{i_1}$. The resulting weighted graph has 16 vertices and 36 edges and it is
depicted in Figure~\ref{fig:RanestadMaster}. There are
five nodes of type
$F_{\aaa}$, namely $F_{0,30,45,55,78}=(1,1,1,1,1)$,
$F_{0,30,45,78}=(1,1,1,0,1)$, $F_{0,30,45,55}=(1,1,1,1,0)$,
$F_{0,30,45}=(1,1,1,0,0)$ and $F_{0,30,78}=(1,1,0,0,1)$. They correspond to the five unlabeled nodes
in the picture.  
\begin{figure}[ht]
    \centering
    \includegraphics[scale=0.4]{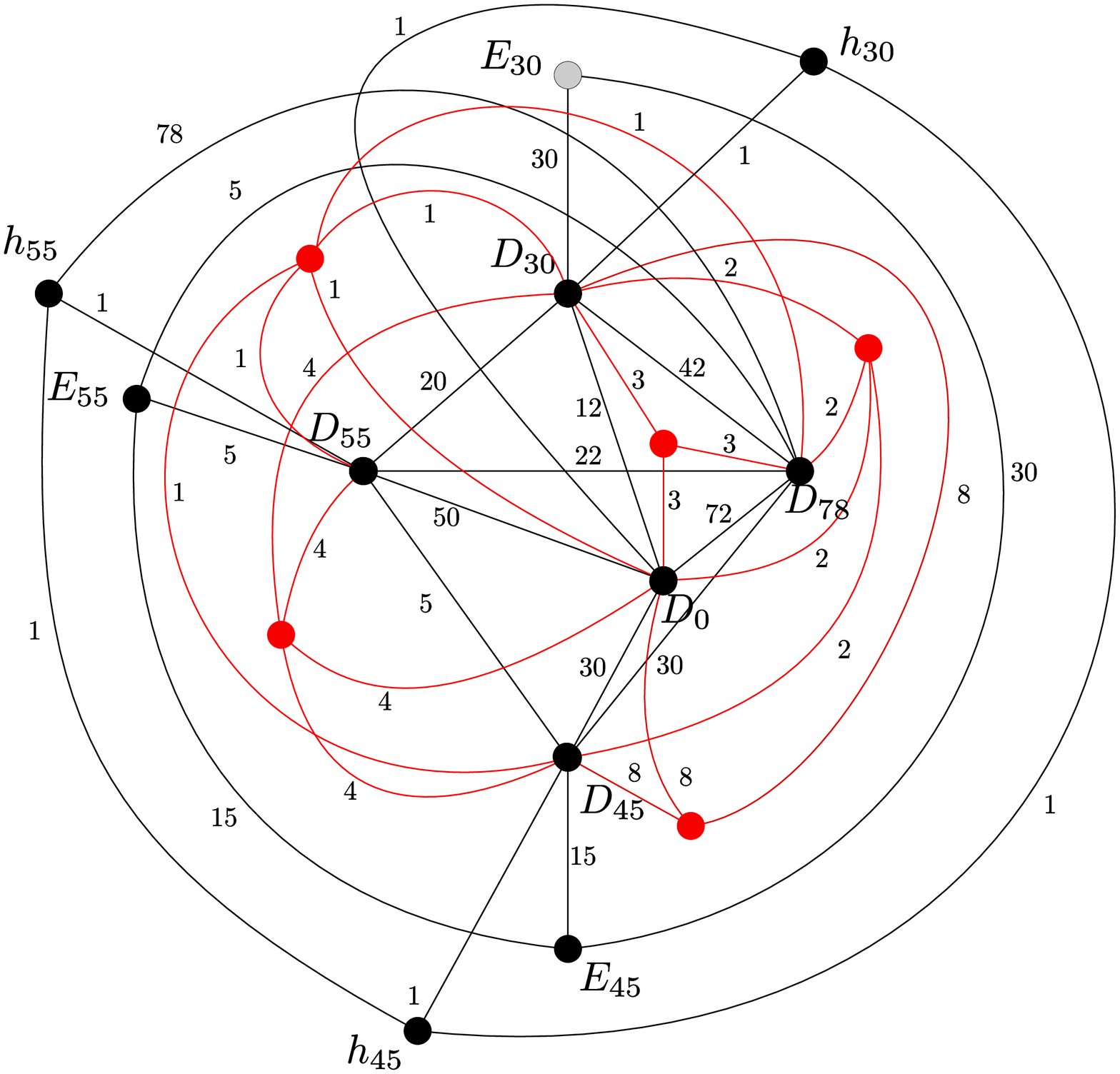}
    \caption{The master graph associated to the curve $(1:t^{30}:t^{45}:t^{55}:t^{78})$.}
    \label{fig:RanestadMaster}
  \end{figure}
\end{example}
Before stating the main result of this section, we recall the
definition of a balanced graph.
\begin{definition}\label{def:balancing} Let $(G,m)\subset \RR^{N}$ be a weighted graph where each node has integer coordinates. Let $w$ be a node in $G$
  and let $\{w_1, \ldots, w_r\}$ be the set of nodes adjacent to
  $w$. Consider the primitive lattices $\Lambda_{w} = \RR\langle
  w\rangle \cap \ZZ^N$ and $\Lambda_{w,w_i}=\RR\langle w, w_i\rangle
  \cap \ZZ^{N}$. Then, $\Lambda_{w,w_i}/\Lambda_{w}$ is a rank one
  lattice and it admits a unique generator which lifts to an element
  in the cone $\RR_{\geq 0}\langle w, w_i\rangle\subset \RR^{N}$. Let
  $u_{w_i|w}$ be one such lifting. We say that the \emph{node} $w$ is
  \emph{balanced} if $\sum_{i=1}^r m_{w_i,w} u_{w_i|w} \in
  \Lambda_w
  $.  If all nodes of $G$ are balanced, then we say that the weighted
  graph $(G,m)$ satisfies the \emph{balancing condition}.
\end{definition}

\begin{theorem}\label{thm:tropSecGraph} The master graph satisfies the \emph{balancing condition}.
\end{theorem}

 \begin{proof}
We proceed by analyzing the balance at each node, following
   Definition~\ref{def:balancing}. The main difficulty lies in finding
   the corresponding vector $u_{w_i|w}$ for each edge $w_iw$ in the
   graph.  We define $g_{i_j}:=\gcd(i_1, \ldots, i_j)$ and
   $g^{i_j}:=\gcd(i_j, \ldots, i_n)$. Note that these are the weights
   $m_{D_{i_j}, E_{i_j}}$ and $m_{D_{i_j}, h_{i_{j}}}$ of the master
   graph. To simplify notation, we set
   $E_{i_{0}}=E_{i_n}=h_{i_{0}}=h_{i_{n}}={\bf 0}$, add the edges
   ${E_{i_{0}} E_{i_{1}}}$, ${E_{i_{n-1}} E_{i_n}}$, ${h_{i_{0}}
     h_{i_{1}}}$ and ${h_{i_{n-1}} h_{i_{n}}}$ to our graph and assign
   weight zero to them.

   We start by checking the balance at all nodes $E_{i_j}$, for $1\leq
   j \leq n-1$.  In this case, we know that $\Lambda_{E_{i_j}}=\ZZ
   \langle E_{i_j}/g_{i_j}\rangle = \ZZ\langle (0, i_1/g_{i_j},
   \ldots, i_j/g_{i_j}, \ldots, i_j/g_{i_j})\rangle$ and
   $\Lambda_{ E_{i_j}, D_{i_j}}=\ZZ\langle E_{i_j}/g_{i_j},
   e_j\rangle$. So $u_{D_{i_j}|E_{i_j}}= e_j$. Similarly, we have
   $u_{D_{i_n}|E_{i_{n-1}}}=e_n$. On the other hand, we know that the lattice
$\Lambda_{E_{i_j}, E_{i_{j+1}}}$ equals $\RR\langle
E_{i_{j+1}}/g_{i_{j+1}},$ $E_{i_j}/g_{i_j}\rangle \cap \ZZ^{n+1}$.  
   By definition, we need to
   extend the primitive vector $E_{i_j}/g_{i_j}$ to a basis of
   $\Lambda_{ E_{i_j}, E_{i_{j+1}}}$ by adding a single vector with
   appropriate sign. In this case, $u_{E_{i_{j+1}}|
     E_{i_j}}=\sum_{k=j+1}^n e_k$.

   Next, we compute $u_{E_{i_{j-1}}|E_{i_j}}$. Here, $\Lambda =
   \ZZ\langle E_{i_{j-1}}/g_{i_{j-1}}, E_{i_j}/g_{i_j}\rangle$ is not
   a primitive lattice, and we need to extend $E_{i_j}/g_{i_j}$ to a
   basis of its saturation $\Lambda_{E_{i_{j-1}},E_{i_j}}$. Our first
   candidate vector is $(E_{i_j}-E_{i_{j-1}})/(i_j-i_{j-1})=
   -\sum_{k=j-1}^n e_k$. However, since $g_{i_{j-1}}$ need not equal
   $g_{i_j}$, we need to slightly modify our choice. We set $v=
   (0, a\, i_1/g_{i_{j-1}}, \ldots, a\, i_{j-1}/g_{i_{j-1}}, -b,
   \ldots, -b)$, for $a,b\in \ZZ$ such that $a\, i_{j} + b\, g_{i_{j-1}}=
   g_{i_j}$. We can check that all nonzero $2\times 2$-minors of the matrix with
   rows $E_{i_j}/g_{i_j}$ and $v$ are of the form: $-b\,i_k/g_{i_j}
   -(i_ji_k)/(g_{i_j} g_{i_{j-1}})=-i_k/g_{j-1}$, for $j-1\leq k \leq
   n$. Therefore,  their gcd equals one, and $E_{i_j}/g_{i_j}$ and $\pm
   v$ generate a primitive lattice which contains $E_{i_{j-1}}$ by
   construction. To determine the correct choice of sign for $\pm v$,
   we write $E_{i_{j-1}}$ as a linear combination of $E_{i_j}/g_{i_j}$
   and $v$, and we require the coefficient of $v$ to be positive. In this
   case,
\[
E_{i_{j-1}}= g_{i_j}(1-a\, (i_j-i_{j-1})/g_{i_j})\cdot E_{i_j}/g_{i_j}
+ g_{i_{j-1}}\, (i_j-i_{j-1})/g_{i_j} \cdot v.
\]
Thus, we conclude that $u_{E_{i_{j-1}}|E_{i_j}}=v$ for $j-1\leq k \leq
n$.  With these
weights, it is straight-forward to check that the graph is balanced at
$E_{i_j}$.

Working with $g^{i_j}$ instead of $g_{i_j}$, a similar procedure to
the one we just described proves that the graph is balanced at the
nodes $h_{i_j}$ for all $1\leq j \leq n-1$. Balance at the nodes
$F_{\aaa}$ follows by construction, so it remains to check the balance
at the nodes $D_{i_j}$. In this case,
$u_{E_{i_j}|{D_{i_j}}}=E_{i_j}/g_{i_j}$,
$u_{h_{i_j}|{D_{i_j}}}=h_{i_j}/g^{i_j}$,
$u_{F_{\aaa}|{D_{i_j}}}=F_{\aaa}$ (for $i_j\in \aaa$), and
$u_{E_{i_{n-1}}|D_{i_n}}=E_{i_{n-1}}$. The balancing equation 
is
\[
\sum_{\aaa \ni i_j} (\sum_r \varphi(r)) F_{\aaa}
+E_{i_j} + h_{i_j} = \sum_{k=0}^n \big(\sum_{r\, \mid
  \,|i_k-i_j|}\!\!\!\!\varphi(r) \,- |i_k-i_j|\big) \, e_k.
\]
Since $\sum_{l\mid s, l>0} \varphi(l)=s$, we conclude that the graph
is also balanced at $D_{i_j}$. %
 \end{proof}

\section{The master graph is a tropical surface}
\label{sec:tropical-geometry}
In this section, we explain the suggestive name ``tropical secant
surface graph'' for the master graph. More concretely, we show that it
is the tropicalization of a surface 
parameterized by the map $(\lambda, w) \mapsto (1-\lambda,
w^{i_1}-\lambda, \ldots, w^{i_n}-\lambda)$.  Before that, we review
the basics of tropical geometry.

\begin{definition} \label{def:tropicalization}
  Given an affine variety $X\subset \CC^{N}$ with defining ideal
  $I=I(X)$, we define the \emph{tropicalization} of $X$ to be the set
\[
\cT X =\cT I=\{w\in \RR^{N}: \init_w(I) \textrm{ does not contain a monomial}\}.
\]
Here, $\init_w(I)=\langle \init_w(f): f\in I\rangle$, and if
$f=\sum_{\alpha} c_\alpha\, \underline{x}^{\alpha}$
, then $\init_w(f)=\sum_{\alpha\cdot w =W}
c_{\alpha}\, \underline{x}^{\alpha}$, where $W=\min\{\alpha\cdot w:
c_{\alpha}\neq 0\}$.  If 
$X\subset \pr^N$, then its tropicalization is defined as
$\cT X'\subset \RR^{N+1}$, where $X'$ is the affine cone over $X$ in
$\CC^{N+1}$.
\end{definition}

Although it may not be clear from
Definition~\ref{def:tropicalization}, tropicalizations are toric in
nature. More precisely, let $\TP^N = (\CC^*)^N$ be the algebraic
torus. Let $Y$ be a subvariety of $\TP^N$, also known as a \textit{very
  affine variety}.  Suppose $I_Y \subseteq \CC[\TP^N] = \CC[y_1^{\pm},
\ldots, y_N^{\pm}]$ is the defining ideal of $Y$. We define the
tropicalization of $Y\subset \TP^N$ as
\[
\cT Y = \{ v \in \RR^N : 1 \notin \init_v(I_Y)\}.
\]
Here, the initial ideal with respect to a vector $v$ is the same as
that in Definition~\ref{def:tropicalization}. Consider the Zariski
closure $\overline{Y}$ of $Y$ in $\CC^N$. It is easy to see that $\cT
Y$ equals $\cT \overline{Y}$.  Indeed, this follows from the fact that
$I_Y$ is the saturation ideal $\big(I_{\overline{Y}} \CC[\TP^N]:
(y_1\cdots y_N)^{\infty}\big)$ and $I_{\overline{Y}}= I_Y \cap
\CC[y_1,\ldots, y_N]$. Therefore, if we start with an irreducible
variety $X\subset \CC^N$ not contained in a coordinate hyperplane,
then we can consider the very affine variety $Y=X\cap \TP^N$, which
has the same dimension as $X$.  The tropical variety $\cT Y$ is a pure
polyhedral subfan of the Gr\"obner fan of $I$ and it preserves an
important invariant of $Y$: both objects have the same
dimension~\citep{BG}.

Tropical implicitization is a recently developed technique to approach
classical implicitization problems~\cite{ElimTheory}. For instance,
when $Y$ is a codimension-one hypersurface, $I_Y = \langle g \rangle$
is principal and $\cT Y$ is the union of all codimension one cones in
the normal fan of the Newton polytope of a polynomial $g$, so knowing
$\cT Y$ can help us in finding $g$. But to achieve this, we need to
compute $\cT Y$ without explicitly knowing $I_Y$. We show how to do
this in Section~\ref{sec:comb-monom-curv}.

A point $w \in \cT X$ is called \emph{regular} if $\cT X$ is a linear
space locally near $w$. We can attach a positive integer to each
regular point of the tropical variety that carries information about
the geometry of $X$. More precisely, we define the \emph{multiplicity}
$m_w$ of a regular point $w$ to be the sum of multiplicities of all
minimal associated primes of the initial ideal
$\init_w(I)$~\citep{TropDiscr}. The multiplicity of a maximal cone in
$\cT X$ agrees with the multiplicity at any of its regular points. One
can show that this assignment does not depend on the choice of the
regular point and that with these multiplicities, the tropical variety
satisfies the \emph{balancing condition} \citep[Corollary
3.4]{ElimTheory}.  

In the case of projective varieties, or in general, when we have a
torus action given by an integer lattice $\Lambda$, the tropical
variety $\cT X$ has a \emph{lineality space}, that is, the maximal
linear space contained in all cones of the fan $\cT X$.  We call the
underlying lattice $\Lambda$ the \emph{lineality lattice}.  For
example, the lineality space of a tropical hypersurface $\cT(g)$
equals the orthogonal complement of the affine span of the Newton
polytope of $g$, after appropriate translation to the origin.  The
extreme case is that of a toric variety globally parameterized by a
monomial map with associated integer matrix $A$. Its tropicalization
$\cT X$ is a classical linear space, namely, the row span of $A$.
$\cT X$ coincides with its lineality space as sets with constant
multiplicity \emph{one} \citep{TropDiscr}.

Since the lineality space $L$ is contained in all cones of $\cT X$, we
can quotient the ambient space by this linear subspace, while
preserving the fan structure~\cite{Mega09}. Furthermore, we intersect
this new set with the unit sphere in $\RR^{N+1}/L$ and consider the
underlying weighted polyhedral complex.  For example, if $X$ is a
surface with no non-trivial torus action, then we view $\cT X$ as a
graph in $\Sn^{N}$.

We now realize the master graph as a tropical surface in $\RR^{n+1}$:
\begin{theorem}\label{thm:MasterGraphIsTropical} Fix a primitive
  strictly increasing sequence $(0, i_1, \ldots, i_n)$ of coprime
  integers.  Let $Z$ be the surface in $\CC^{n+1}$ parameterized by
  $(\lambda, \ww) \mapsto (1-\lambda, \ww^{i_1}-\lambda,\ldots,
  \ww^{i_{n}}-\lambda)$. Then, the tropical surface $\cT Z\subset
  \RR^{n+1}$ coincides with the cone over the master graph as weighted
  polyhedral fans, with the convention that we assign the weight
  $m_{D_{i_1},E_{i_1}} + m_{F_{\underline{e}}, D_{i_1}}$ to the cone
  over the edge $D_{i_1}E_{i_1}$ if the ending sequence
  $\underline{e}=\{i_1, \ldots, i_n\}$ gives a node
  $F_{\underline{e}}$ in the master graph.
\end{theorem}
The proof of this statement involves techniques from geometric
tropicalization and resolution of plane curve singularities.
Beautiful combinatorics are involved in its proof, as we show in
Section~\ref{sec:comb-monom-curv}.

\begin{corollary}\label{cor:ModifiedMasterGraphIsTropical}
  With the notation of Theorem~\ref{thm:MasterGraphIsTropical}, the
  weighted graph obtained by identifying the nodes $E_{i_1}$ and
  $F_{\underline{e}}$ in the master graph, and by assigning weight
  $i_1+ m_{F_{\underline{e}}, D_{i_1}}$ to the edge $D_{i_1}E_{i_1}$,
  agrees with the one-dimensional simplicial complex $\cT Z\cap
  \Sn^n$.
\end{corollary}

\section{Combinatorics of Monomial Curves}
\label{sec:comb-monom-curv}

In this section, we compute the tropical variety of the surface $Z$
described in Theorem~\ref{thm:MasterGraphIsTropical}. Our main tool
will be the theory of \emph{geometric tropicalization}, which we now
present.  The crux of this theory is to read off the tropicalization
of subvarieties of tori from the combinatorics of the boundary of a
suitable compactification. We describe this method for parametric surfaces.

Let $f_1, \ldots, f_N$ be Laurent polynomials in $\CC[t_1^{\pm}, t_2^{\pm}]
$ and consider the rational map ${\bf f} \colon \TP^2
\dashrightarrow \TP^N$, ${\bf f}= (f_1, \ldots, f_N)$. For simplicity,
we assume that the fiber of ${\bf f}$ over a generic point of
$Y\subset \TP^N$ is finite.  Our goal is to compute the
tropicalization $\cT Y$ of the closure of the image of the map ${\bf
  f}$ inside the torus, without knowing its defining ideal.  When
the coefficients of $f_1, \ldots, f_N$ are generic with respect to
their Newton polytopes, a method for constructing $\cT Y$ from these
$N$ polytopes was given
in \citep[Theorem 2.1]{NPofImplicitEquation} and proved in \citep[Theorem
5.1]{ElimTheory}. In the non-generic case, this question is more
subtle and has been partially address in~\cite{Implicitizationsurfaces, ElimTheory}. 
For simplicity, we state the result for the case of parametric \emph{surfaces}
although the method generalizes to higher dimensional subvarieties.

\begin{theorem}[Geometric Tropicalization {\citep[\S 2]{HKT},
    \citep[Theorems
    2.5, 2.8]{Implicitizationsurfaces}}] \label{thm:GT}  
  Let $X$ be a dense open subset of $\CC^2$ or $\pr^2$, and $\mathbf{f}\colon
  Y\to \TP^N$ a generically finite Laurent polynomial map of degree
  $\delta$ parameterizing the surface $Y$.  Let $X\subset
  \overline{X}$ be any normal and $\QQ$-factorial compactification
  whose boundary $W=W_1 \cup\ldots \cup W_k$ has combinatorial normal
  crossings, that is no three components intersect at a point. Let
  $\Delta_{\overline{X},W}$ be the dual graph of $W$, i.e.\ the graph on $\{1, \ldots, m\}$ defined
  by
 \[
 \{i, j\} \in \Delta_{\overline{X},W} \iff W_{i} \cap
 W_{j} \neq \emptyset.\]

 Define the integer vectors $[W_k] := (\val_{W_k}(f_1), \ldots,
 \val_{W_k}(f_N)) \in \ZZ^N$ ($k=1, \ldots, m$) where
 $\val_{W_k}(f_j)$ is the order of zero/pole of $f_j$ along $W_k$. We
 map the abstract graph $\Delta_{\overline{X},D}$ to a graph in $\RR^N$ by
 sending each vertex $i$ to $[W_i]$ and extending linearly on all
 edges.  Define the multiplicity of the edge $([W_i],[W_j])$ to be
\[
m_{([W_i],[W_j])}= \frac{1}{\delta}(W_i\cdot W_j) \operatorname{ index}\big((\RR\otimes_{\ZZ} \ZZ\langle[W_i],
[W_j]\rangle) \cap \ZZ^N : \langle[W_i],
[W_j]\rangle \big),
\]
where $W_{i} \cdot W_{j}$ denotes the intersection number of these
divisors.

Then, the tropical surface $\cT
 Y$ is the cone over this weighted graph.

\end{theorem}
\begin{remark}\label{rm:SumMultGT}
  With the same notation, the multiplicity of a regular point
  $w$ in
  $\cT Y$ equals the sum of the multiplicities of all maximal cones in
  $\cT Y$ containing $w$. 
\end{remark}

To compute $\cT Y$ using the previous theorems, we need a method to
construct a normal $\QQ$-factorial compactification $\overline{X}
\supset Y$ whose boundary has \emph{combinatorial normal crossings
  (CNC)}.  In words, we requires each number of components of the
divisor $W$ to intersect at the expected dimension.  One method for
producing such a compactification is taking the closure $\overline{X}$
of $X$ in $\pr^2$ and resolving the singularities of the boundary
$\overline{X}\setminus X$ to fulfill the CNC
condition~\cite{Implicitizationsurfaces}. Along the way we record
intersection numbers among the boundary components. These numbers
allow us to compute tropical multiplicities, as stated in
Theorem~\ref{thm:GT}. 

\medskip

In what follows, we describe the resolution process of our binomial
surface $Z$ from Theorem~\ref{thm:MasterGraphIsTropical}.  Roughly
speaking, a full resolution gives us several extra (exceptional)
divisors that yield bivalent nodes in the dual graph. If we
contract these curves with negative self-intersection, we obtain a
singular surface whose boundary divisor has CNC. 

Recall that the surface $Z$ was parameterized by ${\bf f}=(f_0,
\ldots, f_n)\colon \CC^2\to Z$, where $f_j:=\ww^{i_j}-\lambda$ ($0\leq
j \leq n$).  Since geometric tropicalization involves subvarieties of
tori, we restrict the domain of the function ${\bf f}$ to the open set
$X=\CC^2\smallsetminus \bigcup_{j=0}^n (f_j=0)$. 
Our task is to
compactify the space $X$. 

We now explain in full detail the compactification process. First, we
na\"ively compactify $X$ inside $\pr^2$. The components of the
boundary divisor of $X$ are $D_{i_j}=(f_j^h(\ww,\lambda, u)=0)$ and
$D_{\infty}=(u=0)$, where $f_j^h$ is the homogenization of $f_j$ with
respect to the new variable
$u$. Figure
~\ref{fig:chartsInPn} illustrate this process in the case of the
binomial arrangement $X$ associated to the index set $\{0, 30, 45, 55,
78\}$ from 
Example~\ref{ex:RanestadMaster}. The black dots
indicate the intersection of three or more of the corresponding
binomial curves, whereas grey dots indicate
pairwise intersections. 

\begin{figure}[htb]
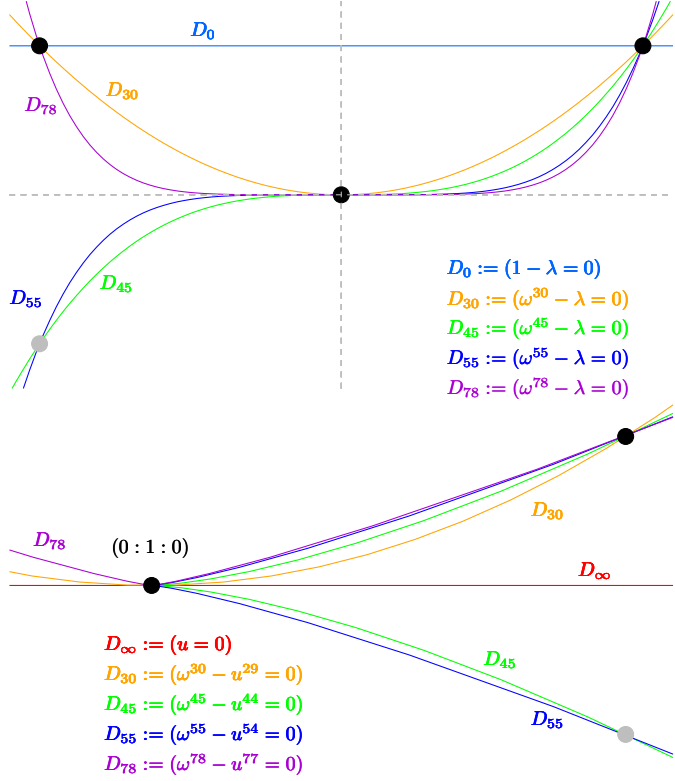

  \centering
  \includegraphics[scale=0.75]{Fig3A.0}

\includegraphics[scale=0.75]{Fig3B.0}
  \caption
  {From top to bottom: 
    $(u=1)$ and $(\lambda=1)$ affine charts describing the
    singularities of the surface $\overline{X}$ in $\pr^2$ given the set
    of exponents $\{0,30, 45,55, 78\}$.}
  \label{fig:chartsInPn}
\end{figure}

The boundary of $X$ in $\pr^2$ encounters three types of
singularities: the origin $(0:0:1)$, the point $(0:1:0)$ at infinity,
and 
singularities in $\TP^2$. We resolve them all by
blow-ups. 
After contracting appropriate exceptional curves, the resolutions
diagrams from Figures~\ref{fig:OriginBlowup}
and~\ref{fig:InfinityBlowup} are \emph{precisely} the graphs on the
left side of Figure~\ref{fig:OriginAndInfinityAndRest}. The nodes
$E_{i_j}$ ($1\leq j\leq n-1$) and $h_{i_j}$ ($2\leq j\leq n-1$)
correspond to exceptional divisors, whereas $h_{i_1}$ refers to the
strict transform of the divisor $D_{\infty}$. All intersection
multiplicities involving the divisors $E_{i_j}$ or $h_{i_j}$ equal
one. Following Theorem~\ref{thm:GT}, we see that the computation of
all multiplicities in the graph $\cT Z$ reduces to calculating
indices of suitable lattices associated to edges of $\cT Z$.

\begin{figure}[htr]
  \centering
  \includegraphics[scale=0.35]{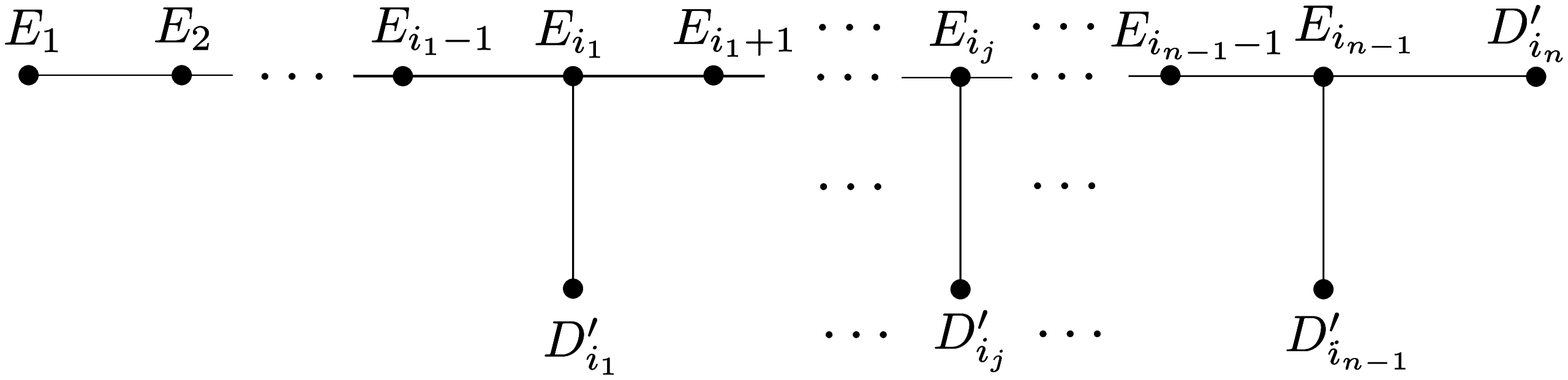}
  \caption{Resolution by blow-ups at the origin, where the $E_{i_j}$'s
    denote exceptional divisors and each $D'_{i_j}$ is the strict transform
    of the boundary divisor $D_{i_j}$.}
  \label{fig:OriginBlowup}
\end{figure}
We now describe the resolution process at the origin, whose local
behavior is illustrated in the top of Figure~\ref{fig:chartsInPn}. At
this singular point, all $n$ curves $D_{i_1}, \ldots, D_{i_n}$
intersect and they are tangential to each other. For any $j$, after a
single blow-up, and after applying the change of coordinates $\lambda
=w \lambda'$, we see that the strict transform of $D_{i_j}$ is
isomorphic to $D_{i_j-1}$ for all $1\leq j \leq n$. This implies that
we can resolve the singularity at the origin after $i_{n-1}$
blow-ups. Moreover, we may use the previous isomorphism to compute the
pull-back of each divisor $D_{i_j}$, one step at a time. For example,
after the first blow-up $\pi_1$, the proper transform of $D_{i_j}$
equals $\pi_1^{*}(D_{i_j})= D_{i_j}' + E_{1}$, where $E_{1}=(w=0)$ is
the exceptional divisor and $D_{i_j}'$ is the strict transform of
$D_{i_j}$. After a second blow-up $\pi_2$, we get $\pi_2^*(E_1)= E_1'
+E_2$, and $\pi_2^*(D_{i_j}')-E_{2}=D_{i_j}''\simeq D_{i_j-2}$, so
$(\pi_1\circ \pi_2)^*(D_{i_j})=D_{i_j}'' + E_{1}' +2E_2$ with
$D_{i_j}''\simeq D_{i_j-2}$ and $E_1'\cdot D_{i_j}''= 0$.  To simplify
notation, we label all exceptional divisors by $E_{l}$ and for each
$j$ we let $D_{i_j}'$ be the strict transform of $D_{i_j}$ under the
composition $\pi$ of all blow-ups.  All exceptional divisors satisfy:
\begin{center}
  \begin{minipage}[c]{0.4\linewidth}
    \[
    E_{l}\cdot E_k =
    \begin{cases}
      1 & \text{ if } |l-k|=1,\\
      0 & \text{ otherwise,}
    \end{cases}
    \]
  \end{minipage}
  \begin{minipage}[c]{0.4\linewidth}
    \[
    D_{i_j}'\cdot E_l =
    \begin{cases}
      1 & \text{ if } l=i_j,\\
      0 & \text{ otherwise.}
    \end{cases}
    \]
  \end{minipage}
\end{center}

\noindent
Proceeding by induction, we obtain
\begin{equation}
\pi^*(D_{i_j}) = D_{i_j}' + \sum_{l=1}^{i_j} l \cdot E_l
+\sum_{l={i_j+1}}^{i_{n-1}} i_j\cdot  E_l \qquad \qquad 1\leq j \leq n.
\label{eq:7}
\end{equation}
By convention, the sum over an empty set equals 0. 

Figure~\ref{fig:OriginBlowup} illustrates the resolution diagram at
the origin. If
we eliminate the bivalent nodes $E_l$ from
Figure~\ref{fig:OriginBlowup} by contracting the corresponding curves
with negative self-intersection, we obtain the graph $G_{D,E}$
depicted in the left side of
Figure~\ref{fig:OriginAndInfinityAndRest}, where, by abuse of
notation, the strict transform of $D_{i_j}$ is also denoted by
$D_{i_j}$. From~\eqref{eq:7} we see that the divisorial valuation of
each exceptional divisor gives the integer vector $E_{i_j}$ described
in Theorem~\ref{thm:tropSecGraph}.
\smallskip

At infinity, the resolution process is more delicate. The toric
arrangement in the corresponding affine chart of $\pr^2$ is depicted
in the bottom of Figure~\ref{fig:chartsInPn}.  Here, the singular
point $p=(0:1:0)$ corresponds to the intersection of $D_{\infty}$ and
all divisors $D_{i_j}$ with $i_j\geq 2$. All prime divisors $D_{i_j}$,
$i_j\geq 2$ have a singularity at $p$, so we first need to perform a
blow-up at this point to smooth them out. More precisely, if $\pi_0$
denotes this blow-up we obtain
\[
\pi_0^*(D_{i_j})=D_{i_j}'+ (i_j-1)H\;,\qquad
\pi_0^*(D_{\infty}) = D_{\infty}' +H,
\]
where $H=(t=0)$ is the exceptional divisor and
$D_{i_j}'=(\ww-t^{i_j-1})\simeq D_{i_j-1}$, $D_{\infty}'=(w=0)$ are
the strict transforms of the corresponding curves.

As the reader may have discovered already, the setting after applying
$\pi_0$ is very similar to the one we described for the singularity at
the origin, although there are some minor local differences between
them 
that are worth pointing out. Firstly, there is a singularity coming
from theintersection of the divisors $D_{i_s-1}, \ldots, D_{i_n-1}$,
where $s$ is the minimum index satisfying $i_s\geq 2$. This singular
point plays the role of the origin in the chart $(u=1)$. In addition,
there are two extra divisors $D'_{\infty}$ and $H$, passing through
this point. These curves had no counterpart at the chart containing
the origin.  Along the resolution, $D_{\infty}'$ is separated from the
other divisors after a single blow-up, but the strict transform of
$H$ is tangential to the strict transform of \emph{all}
divisors $D_{i_j}'$ that meet $H$.

The resolution diagram at infinity is shown in
Figure~\ref{fig:InfinityBlowup}. In that picture,
all exceptional divisors are denoted by $h_{l}$ ($2\leq l\leq i_n$)
and we label the strict transforms of $D_{\infty}, D_{i_j}$ and $H$ by
$D_{\infty}''$, $D_{i_j}''$ and $H'$ respectively.
The pull-backs under the composition $\pi$ of the last
$i_n-1$ blow-ups give:
 \begin{equation*}\left\{
   \begin{aligned}
     \pi^{*}(D_{\infty}')&\!\!\! = D_{\infty}''
     + \sum\limits_{l=2}^{i_n}\;\;  h_l,\qquad  
     \pi^*(H)  \!\!\! =  H' \;\,+ \sum\limits_{l=2}^{i_n} (l-1)\cdot h_l,\\
     \pi^*(D_{i_j}') &\!\!\! = D_{i_j}''\, + \sum\limits_{l=2}^{i_j}
     (l-1) \cdot h_l + \sum\limits_{l=i_j+1}^{i_n} (i_j-1) \cdot
     h_l\qquad (i_j\geq 2).
   \end{aligned}
\right.
\end{equation*}
Composing $\pi$ with the initial blow-up $\pi_0$ at the point $(0:1:0)$, we get:
\begin{equation*}
  \begin{cases}
    (\pi\circ \pi_0)^*(D_{i_j}) &\!\!\! = D_{i_j}'' + (i_j-1) H' + \sum\limits_{l=2}^{i_j}i_j(l-1)
    \cdot h_l + \sum\limits_{l=i_j+1}^{i_n} (i_j-1) l \cdot h_l \quad (i_j\geq 2),\\
    (\pi\circ \pi_0)^*(D_{\infty}) & \!\!\!= 
 H' + D_{\infty}''+ \sum\limits_{l=2}^{i_n} l \cdot h_l.
  \end{cases}
\end{equation*}
All intersection numbers $h_{l} \cdot h_{l+1}$, $h_{i_j}\cdot
D_{i_j}''$, $D_{\infty}''\cdot h_2$ and $h_{i_n}\cdot H'$ equal one,
whereas all other pairs have intersection number zero. In addition, we
know that $D_{\infty}$ intersects $D_0$ at a point, thus
$D_{\infty}''\cdot D_0=1$. Finally, the divisor $D_{\infty}$  also
intersects $D_{i_1}$ at a point different from $(0:1:0)$, only if
$i_1=1$. Thus, $D_{\infty}''\cdot D_{1}=1$ if $i_1=1$ or $0$ in all
other cases. 

\begin{figure}[htr]
  \centering
  \includegraphics[scale=0.4]{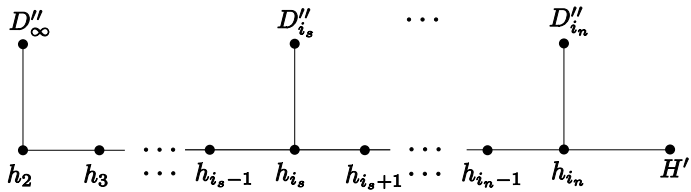}
  \caption{Resolution by blow-ups at infinity. Here, $s$
    is the minimum index with $i_s\geq 2$.}
  \label{fig:InfinityBlowup}
\end{figure}

We now explain the transition from the resolution diagram at infinity
to the graph $G_{h,D}$, depicted at the bottom-left of
Figure~\ref{fig:OriginAndInfinityAndRest}.  As we did when blowing up
the origin, we only keep the $n-1$ exceptional divisors $h_{i_2},
\ldots, h_{i_n}$ giving non-bivalent nodes in the resolution
diagram. In addition, we contract the strict transform $H'$ of the
exceptional divisor $H$, since it has negative self-intersection.

The degree of the node $D_{\infty}''$ in the dual graph is
determined by the value of the index $s$. If $i_1\geq 2$, then $s=1$
and $D_{\infty}''$ is a bivalent node adjacent to $D_0$ and $h_{i_1}$,
so we remove it from the resolution diagram. On the contrary, if
$i_1=1$, then $s=2$ and $D_{\infty}''$ has degree 3: it is adjacent to
the nodes $D_0$, $D_1$ and $h_{i_2}$. The node $h_{i_1}$ in $G_{h,D}$
corresponds to the divisor $D_{\infty}''$.  In both cases, and after
removing all bivalent nodes and the node associated to $H'$, we get
the graph $G_{h,D}$.  \medskip

We now study multiple intersections between divisors in $\TP^2$. If
$(\omega, \lambda)$ satisfies $f_j = \omega^{i_j}-\lambda = 0$ and
$f_k=\omega^{i_k}- \lambda = 0$, then $\omega^{i_j} = \lambda =
\omega^{i_k}$, so $\omega$ is a primitive $r^{\text{th}}$ root of
unity for some $r\mid (i_k-i_j)$. Equivalently, $i_j \equiv i_k \equiv
u \pmod{r}$, $\omega= e^{2\pi ip/r}$ and $\lambda = \omega^u$ for some
integer $p$ coprime to $r$. All other curves $(f_l =0)$ with $i_l
\equiv u \pmod{r}$ also meet at $(\omega, \lambda)$. We represent this
crossing point by $x_{p,r,u}$ and the indices of all curves meeting at
$x_{p,r,u}$ by $\aaa_{r,u}$, or $\aaa$ for short. That is,
$$
x_{p,r,u}=( e^{2\pi ip/r}, e^{2\pi ipu/r}), \quad \aaa=\aaa_{r,u} :=
\{i_j \,\mid \, i_j \equiv u \,\,(\mbox{mod }r)\}.
$$
Furthermore, the gradients of the curves meeting at the
point $x_{p,r,u}$ are pairwise independent, so the curves intersect
transversally at $x_{p,r,u}$.

If three or more curves meet at a point $x_{p,r,u}$ in $\TP^2$, we
blow up this point to separate the curves.  After a single blow-up, 
we obtain a new exceptional divisor $F_{\aaa, x_{p,r,u}}$ which
intersects the strict transform of all ${D}_{i_j}$ ($i_j\in \aaa$)
with multiplicity one.  The resolution diagram is the graph
$G_{F_{\aaa},D}$ on the right-hand side of
Figure~\ref{fig:OriginAndInfinityAndRest}, where we identify the
node $F_{\aaa}$ with the corresponding divisor $F_{\aaa,
  x_{p,r,u}}$. From these intersection numbers, we conclude that the
divisorial valuation of the exceptional divisor $F_{\aaa, x_{p,r,u}}$
gives $[F_{\aaa,x_{p,r,u}}]= \sum_{i_j\in\aaa} e_j$ for all
intersection points $x_{p,r,u}$ coming from the same subset $\aaa$.
Thus, we get a single integer vector $F_{\aaa} = \sum_{i_j \in \aaa}
e_j$ in the realization of the dual graph, as desired. This
explains the notation chosen for the graph $G_{F_{\underline{a}},D}$
in Figure~\ref{fig:OriginAndInfinityAndRest}, where we accounted
only for the indices of divisors intersecting at a point, rather than
recording the point $x_{p,r,u}$ itself. To simplify the computation of
multiplicities in the tropical surface $\cT Z$, we also blow up
crossings with $|\aaa| = 2$. Such blow-ups give bivalent nodes
that we can easily discard in the end.
\medskip

Next, we compute the divisorial valuations of all boundary components
in our compactification. 
%
%
First, we extend the original parameterization of $Z$ from $\TP^2$ to
$\pr^2$. 
The extended map is defined
as 
\[{\bf f}\colon X\subset \pr^2 \to Z\cap \TP^{n+1} \quad
{\bf f}(\ww,\lambda,u)  = \big(\frac{u-\lambda}{u} ,
\frac{\ww^{i_1} -\lambda u^{i_1-1}}{u^{i_1}}, \ldots, \frac{\ww^{i_n}-\lambda
u^{i_n-1}}{u^{i_n}}\big),
\] that is,
${\bf f}(\ww,\lambda, u)= (f_j^h(\ww, \lambda,
u)/u^{\deg(f_j)} )_{i=0}^{n}$.  We compose ${\bf f}$ with the
resolution $\pi$ 
to get the map ${\bf \tilde{f}}=\pi \circ {\bf f}\colon \tilde{X}\dashrightarrow Z\cap
\TP^{n+1}$. 

We know that the functions ${ f_0}, \ldots, { f_{n}}$ are units
on ${X}$ and rational functions on the closure of $X$ in $\pr^2$: they
are pullback under ${\bf{f}}$ of the characters of $\TP^{n+1}$. In
particular, they have zeros and poles only along the boundary of
$\overline{X}$.  By the universal property of the blow-up, the same
holds for $\tilde{X}$ and the functions ${ \tilde{f}_1}, \ldots,
{ \tilde{f}_n}$, since they are pullback under ${\bf \tilde{f}}$ of
the characters of $\TP^{n+1}$. Therefore,
 \[
({ \tilde{f}_0}) = \pi^*({
  f_0})=\pi^*( D_{i_0}-D_{\infty})\quad \text{and}\quad ({ \tilde{ f}_j})=
\pi^*(D_{i_j} - i_j D_{\infty})\quad\text{for }j\geq 1.
\]
For simplicity and to agree with the notation of the graphs in
Figure~\ref{fig:OriginAndInfinityAndRest}, we denote strict
transforms of all divisors with the label of the corresponding
original divisors.  With this convention,
$({ \tilde{f}_j})$ equals
\begin{equation*}
\begin{cases}
    D_{i_j} +  \! \sum\limits_{l=1}^{i_j} l \!\cdot\! E_l
+ \!\sum\limits_{l={i_j+1}}^{i_{n-1}} \! i_j\cdot  E_l -
\;\;\;\;D_{\infty} - H -\sum\limits_{l=2}^{i_j} l \!\cdot\! h_l
\,-\!\sum\limits_{l=i_j+1}^{i_n} l \!\cdot\! h_l  
+\!   \sum\limits_{\substack{\aaa \ni i_j\\ x_{p,r,u}}}
 \! F_{\aaa, x_{p,r,u}}  &  \text{ if } i_j< 2,\\
     D_{i_j}  + \! \sum\limits_{l=1}^{i_j}\! l \!\cdot E_l
+  \!\sum\limits_{l={i_j+1}}^{i_{n-1}}   i_j\!\cdot\!  E_l   -i_j
\!\cdot\! D_{\infty} - H -
 \sum\limits_{l=2}^{i_j}  i_j \!\cdot\! h_l
-\!\sum\limits_{l=i_j+1}^{i_n}\! l \!\cdot\! h_l +
 \! \sum\limits_{\substack{\aaa \ni i_j\\ x_{p,r,u}}}
 \! F_{\aaa,
       x_{p,r,u}} &  \text{ else }.
  \end{cases}
\end{equation*}
The corresponding divisorial valuations are read off from the
columns of the matrix of coefficients of $(\tilde{f}_j)_{j=1}^n$
with respect to the divisors $D_{i_j}, E_{i_j}, h_{i_j}$, $H$ and
$F_{\aaa, x_{p,r,u}}$.  Using Theorem~\ref{thm:GT}, we get
the following nodes in the tropical variety $\cT Z$:
\begin{equation}\label{eq:8}
  \left\{\begin{aligned}
\  \!
    [D_{i_j}] &=e_{j}\;
,\;\; [H ]= -{\bf 1} 
\;,\;\;
    [F_{\aaa, x_{p,r,u}}] =\sum\limits_{i_j\in \aaa} e_j
 \;,\;\;
[D_{\infty}]
= - \sum_{
     i_j<2}e_j - \sum_{ i_j\geq 2} i_j \cdot e_j\;
,\;\;\\
 [E_l] &= \sum_{
      l\leq i_j} l \cdot e_j + \sum_{ l > i_j} i_j \cdot e_j
   \quad \qquad\;\;(1\leq l\leq i_{n-1})
\;
,\;\;\\
    [h_l]&= -\sum\limits_{ i_j<l}l \cdot e_j - \sum\limits_{l\leq i_j}
    i_j \cdot e_j
  \qquad  \;\;(2\leq l\leq i_n)\;.
  \end{aligned}\right.
\end{equation}
We see that $[h_{i_n}]= i_n [H]$, so the cone over the edge $h_{i_n}H$
in the realization of the dual graph is one-dimensional. This explains
why we do not see the divisor $H$ in the graph $G_{h,D}$ from
Figure~\ref{fig:OriginAndInfinityAndRest}. Likewise $i_1\cdot [F_{i_1,
  \ldots, i_n}] = [E_{i_1}]$ if $\gcd(i_n-i_1, \ldots, i_2-i_1)\neq
1$, so the cones over the edges $F_{i_1, \ldots, i_n}D_{i_n}$ and
$E_{i_1}D_{i_1}$ agree. In this case, we replace these two cones by a
single cone
, adding the two weights.  However, these are not the
only identifications we can perform to simplify our construction. 
The next result implies that we can eliminate the bivalent nodes
$E_{l}, h_l$ ($l\neq i_j$) as well as the nodes $h_{i_n}$ and
$D_{\infty}$ from this graph. Roughly speaking, it says that the
bivalent nodes $E_{i_l}$ and $h_{i_l}$ are contained in the
two-dimensional cones spanned by the corresponding nodes $E_{i_j},
E_{i_{j+1}}$ and $h_{i_j}h_{i_{j+1}}$ with $i_j<l<i_{j+1}$, and
similarly for $D_{\infty}$. It also asserts that there are no overlaps
between cones over the edges other than the one we already mentioned.
Using these two facts we can reduce our resolution graphs to
$G_{E,D}$, $G_{h,D}$, and $G_{F_{\aaa},D}$, thus proving the set
theoretic equality in Theorem~\ref{thm:MasterGraphIsTropical}. Recall
that $s$ is the minimum index such that $i_s\geq 2$.
\begin{lemma}\label{lm:coarseningTropicalgraph}
With the notation of~\eqref{eq:8}, 
the following  have equalities hold:
\begin{enumerate}\item 
$\RR_{\geq
    0} \langle [E_{l}], [E_{l+1}]\rangle \bigcap \RR_{\geq 0}
\langle  [E_{l+1}], [E_{l+2}] \rangle = \RR_{\geq 0}\langle
[E_{l+1}]\rangle $ ($i_j\leq l
  \leq i_{j+1}-2$, $0<j<n-1$);
\item
$  \RR_{\geq 0} \langle [E_{i_j}], \ldots, [E_{i_{j+1}}]\rangle =
\RR_{\geq 0}\langle [E_{i_j}], [E_{i_{j+1}}]\rangle$ ($1\leq j \leq {n-2}$);
\item
$\RR_{\geq 0}
\langle [h_{l}], [h_{l+1}]\rangle \bigcap \RR_{\geq 0} \langle [h_{l+1}], [h_{l+2}]
\rangle = \RR_{\geq 0}\langle [h_{l+1}] \rangle $ ($2\leq i_j\leq l
  \leq i_{j+1}-2$, $0<j<n$);
\item $\RR_{\geq 0} \langle [h_{i_j}], \ldots, [h_{i_{j+1}}]\rangle =
\RR_{\geq 0}\langle [h_{i_j}], [h_{i_{j+1}}]\rangle$ ($1\leq j \leq {n-1}$);
\item  $[h_{i_n}]\in \RR_{\geq 0}\langle [h_{i_{n-1}}], [D_{i_n}]\rangle$;

\item $\RR_{\geq
  0}\langle [E_{1}], \ldots, [E_{i_1}]\rangle =\RR_{\geq 0}\langle
[E_{i_1}]\rangle $ and $\RR_{\geq 0}\langle [h_{2}], \ldots, [h_{i_s}]\rangle = 
\RR_{\geq 0}\langle [h_{2}], [h_{i_s}]\rangle$;
\item If $s=1$, then 
$[D_\infty]=[h_{i_1}]$;
  if $s=2$, then 
  $\RR_{\geq 0}\langle [D_{\infty}],  [h_{i_1}], [h_{i_2}]
\rangle = 
\RR_{\geq 0}\langle [h_{i_1}] , [h_{i_2}]\rangle$.
\end{enumerate}
Moreover, among maximal cones over the master graph, there are no
two-dimensional intersections except when $F_{\underline{e}}$ is a
node in the master graph where $\underline{e}=\{i_1, \ldots, i_n\}$.
In this case, $i_1[F_{\underline{e}}]=[E_{i_1}]$ and $\RR_{\geq 0}
\langle [F_{\underline{e}}], [D_{i_1}]\rangle = \RR_{\geq 0} \langle
[E_{i_1}], [D_{i_1}]\rangle$.
\end{lemma}
\begin{proof}
  We prove the identities involving the rays $[E_{l}]$ ($1\leq l \leq
  i_{n-1}$) in $(i)$ and $(ii)$. The claims for $[h_l]$ in $(iii)$ and
  $(iv)$ can be proven analogously. Assume $i_j\leq l\leq
  i_{j+1}-2$. Then $[E_{l+1}]=[E_l] + \sum_{k\geq j+1} e_k$ and
  $[E_{l+2}]=[E_l] + 2\sum_{k\geq j+1} e_k$, and the first identity
  follows by simple linear algebra arguments.

To prove the second claim, it suffices to show that $[E_l]\in \RR_{\geq
  0}\langle [E_{i_j}], [E_{i_{j+1}}]\rangle$ if $i_j< l <i_{j+1}$. In
fact, by linear algebra calculations, we obtain 
$[E_l]= \frac{i_{j+1}-l}{i_{j+1}-i_j} \cdot
[E_{i_{j}}]+\frac{l-i_{j}}{i_{j+1}-i_j} \cdot [E_{i_{j+1}}]$. 
The identities in $(vi)$ are a direct consequence of the equalities
$[E_{l}] =l\sum_{j\geq 1}e_j=\frac{l}{i_1}[E_{i_1}]$, and
$[h_l]=\frac{i_s-l}{i_s-2} [h_2] + \frac{l-2}{i_s-2} [h_{i_s}]$ for
all $2\leq l \leq i_s$.

To prove $(vii)$ we consider all pairs of maximal cones and
compute their intersection. We get
either the origin or the cone over a node in the master graph. 
\end{proof}


Next, we compute the weights of all edges in the
$\cT Z$ 
using
Theorem~\ref{thm:GT} and the map $\mathbf{\tilde{f}}$. 
From the 
resolution 
$\tilde{X}$, we know that the intersection number of any two
boundary  curves is zero or one. Using
Lemma~\ref{lm:coarseningTropicalgraph}, we see that there are no
two-dimensional overlaps, 
except for the cones over the edges $D_{i_1}E_{i_1}$ and $F_{i_1,
  \ldots, i_n}D_{i_1}$. The degree of the map $\mathbf{\tilde{f}}$ is one.

With the exception of the edge $D_{i_1}E_{i_1}$, the formula for
computing weights on the edges containing $D_{i_j}$, $h_{i_j}$,
$E_{i_j}$ 
involves a single summand, namely the corresponding
lattice index. This number is the gcd of the $2\times 2$-minors of a
matrix whose rows are the two nodes of each edge, and it agrees with the
weights assigned to the master graph.

To end, we obtain the multiplicity of the cones over the edges
$F_{\aaa} D_{i_j}$ in $\cT Z$, with $\aaa\neq \{i_1, \ldots,
i_n\}$. In this case, all summands in the formula equal one and so the
multiplicity equals the number of summands.  The summands are in
one-to-one correspondence with the crossing points $x_{p,r,u}$, where
$\aaa=\{i_j\, \mid \, i_j \equiv u\,\,\mbox{\text{mod }r}\}$ and $p$
is coprime to $r$. Therefore, the number of summands is $\sum_{r}
\varphi(r)$, where the sum is over all possible common differences $r$
of arithmetic sequences giving the same set $\aaa$.
Finally, if $\underline{e}=\{i_1, \ldots, i_n\}$ gives a node
$F_{\underline{e}}$ in the master graph, the divisors $E_{i_1}$ and
$F_{\underline{e}}$ map to proportional rays $[E_{i_1}]$ and
$[F_{\underline{e}}]$.  The multiplicities of the cones over the edges
$F_{\underline{e}}D_{i_j}$ with $j\geq 2$ equal $\sum_r \varphi(r)$
for all common differences $r$ generating the set $\underline{e}$. The
formula to compute the weight of the edge $F_{\underline{e}}D_{i_1}$
has an extra summand: the one involving the term
$E_{i_1}D_{i_1}$. Hence, $F_{\underline{e}}D_{i_1}$ has weight
$m_{D_{i_1},E_{i_1}} + m_{F_{i_1, \ldots, i_n}, D_{i_1}}= i_1 +\sum_r
\varphi(r) $. This concludes our proof of
Theorem~\ref{thm:MasterGraphIsTropical}.


\section
  {The master graph under Hadamard products} 
\label{sec:trop-secant-graph}

In this section, we use the 
master graph to construct a new weighted graph: the \emph{tropical
  secant graph}.  This graph encodes the tropicalization of the first
secant variety of a monomial projective curve $C$ parameterized by
$(1: t^{i_1}: \ldots: t^{i_n})$, where $0= i_0 \leq i_1 \leq \ldots
\leq i_n$ are integers.  We define the first secant variety of the
curve $C$ as
\begin{equation*}
Sec^1(C)=\overline{\{a \cdot {p} + b \cdot {q}\,\mid\,
   {p}, {q}\in C, (a:b)\in \pr^1\}}\subset \pr^n.
\end{equation*}

As discussed in Section~\ref{sec:tropical-geometry}, tropicalizations
are toric in nature. Thus, for the rest of this section, instead of
looking at the projective varieties $C$ and $Sec^1(C)$, we study the
corresponding very affine varieties obtained by intersecting their
affine cones in $\CC^{n+1}$ with the torus $\TP^{n+1}$. To simplify
notation, we also denote them by $C$ and $Sec^1(C)$. The
tropicalizations of the projective varieties and their corresponding
very affine varieties are the same, but we think of the projective one
as living in the tropical projective torus
$\TP\pr^{n}:=\RR^{n+1}/\RR\,\langle \mathbf{1}\rangle$ rather than in $\RR^{n+1}$,
reducing its dimension by one. We parameterize the secant variety by the
\emph{secant map} 
\begin{equation}
\phi\colon \TP^{4} \dashrightarrow
\TP^{n+1}, \quad \phi(a,b,s,t) = (as^{i_k}+bt^{i_k})_{0\leq k \leq n}
.\label{eq:SecantMap}
\end{equation}
After a monomial change of coordinates $b = -\lambda a$ and $t =
\omega s$, we rewrite $\phi$ as 
\begin{equation}
\phi(a,s,\omega,\lambda) = \big(as^{i_k} \, (\omega^{i_k} -
\lambda)\big)_{0\leq k \leq n}.\label{eq:2}
\end{equation}
From this observation, it is natural to consider the Hadamard product of subvarieties of tori:

\begin{definition}\label{def:star}
  Let $X,Y\subset \TP^{N}$ be subvarieties of tori.  The
  \emph{Hadamard product} of $X$ and $Y$ equals
\[
X \centerdot Y = \overline{\{(x_1y_1, \ldots,
    x_{N}y_{N}) \, |\, x\in X, y\in Y\}}\subset \TP^{N}.
\]
\end{definition}

From the construction, we get the following characterization of our
secant variety, where $Z$ is precisely the surface $Z$ from Theorem~\ref{thm:MasterGraphIsTropical}.
\begin{proposition}\label{pr:ReparamSecant}
   Let $C$ be the monomial curve $(1:t^{i_1}: \ldots:
  t^{i_n})$ and let $Z$ be the surface parameterized by $(\lambda, \omega)
  \mapsto (1-\lambda, \omega^{i_1} - \lambda, \ldots, \omega^{i_n} -
 \lambda)$. Then, the first secant variety $Sec^1(C)\subset \TP^{n+1}$
  is the Hadamard product $C \centerdot Z$.
\end{proposition}
We now explain the relationship between Hadamard products and their
tropicalization:
\begin{proposition}\label{pr:TropOfSecant} \citep[Corollary 11]{Mega09}
  Given $C,Z$ as in Proposition~\ref{pr:ReparamSecant}, then as
  \emph{sets}
  \begin{equation*}
\cT Sec^1(C) =\cT 
C + \cT Z,
\end{equation*}
where the sum on the right-hand side denotes the Minkowski sum in
$\RR^{n+1}$.
\end{proposition}
Since the curve $C$ is parameterized by monomials, its tropicalization
$\cT C$ is the two-dimensional vector space spanned by the lattice vectors
$\{ (1, \ldots, 1), (0, i_1, \ldots, i_n)\}$, with constant weight one. In addition, $\cT Z$ is a
pointed polyhedral fan, and the lineality space of $\cT Sec^1(C)$ is
$\cT C$. Thus, the associated spherical complex $(\cT \Sec^1(C)/\cT
C)\cap \Sn^{n-2}$ is a graph. It can be obtained by identifying nodes
and edges in the master graph by their residue class modulo $\cT
C$. In the remainder of this section we explain this reduction process.

Before diving into the computation of $\cT Sec^1(C)$ for any
projective monomial curve $C$, we show that it suffices to treat the
case of exponent vectors that are primitive and whose coordinates are
all distinct. This simplifies the combinatorics encoded in the
multiplicities as well.  Here is the precise statement:
\begin{lemma}\label{lm:PrimitiveIncreasingExponentVector}
  Via reparameterizations, we can assume that the exponent vector
  parameterizing the curve $C$ is a \emph{primitive}
  lattice vector $(0, i_1, \ldots,
  i_n)$ with $0=i_0<i_1<\ldots<i_n$.
\end{lemma}
The first claim follows by reparameterizing the curve $C$ as $t
\mapsto (1:t^{\frac{i_1}{g}}: \ldots: t^{\frac{i_n}{g}})$, where
$g=\gcd (i_1, \ldots, i_n)$.  The second assertion is a direct
consequence of the following result that shows the interplay between
maps on tori and their tropicalization.  Let $\alpha\colon \TP^r \to
\TP^N$ be a monomial map whose exponents are encoded in a matrix $A\in
\ZZ^{N\times r}$.
\begin{theorem}\citep[Theorem 3.12]{ElimTheory}
\label{thm:ST}
Let $V \subset \TP^r$ be a subvariety.  Then, $ \cT(\alpha(V)) = A(\cT
V).  $ Moreover, if $\alpha$ induces a generically finite morphism of
degree $\delta$ on $V$, the multiplicity of a regular point $w$ of
$\cT (\alpha(V))$ is
\begin{equation}
m_w = \frac{1}{\delta} \cdot \sum_v m_v \cdot \text{ index
}(\mathbb{L}_w \cap \ZZ^N : A(\mathbb{L}_v \cap \ZZ^r)),\label{eq:ST}
\end{equation}
where the sum is over all points $v \in \cT V$ with $A v = w$.  We
also assume that the number of such $v$ is finite and that all of them are
regular points in $\cT V$. In this setting, $\mathbb{L}_v, \mathbb{L}_w$
denote the linear span of neighborhoods of $v \in \cT V $ and $w \in
A (\cT V)$ respectively.
\end{theorem}

\begin{proof} [\textbf{Proof of
    Lemma~\ref{lm:PrimitiveIncreasingExponentVector}}] Let $\{0, i_1,
  \ldots, i_r\}$ be the distinct values in the exponent vector
  defining the curve $C$ in increasing order. We partition the set of
  indices $\{0, \ldots, n\}$ into $S_0\sqcup \ldots \sqcup S_r$, where
  each $S_j$ consists of all indices with the same value $i_j$.  The 
map
  \[
  \alpha\colon \TP^r \to \TP^{n+1}\quad (t_1, \ldots, t_r) \mapsto
  (\underbrace{1, \ldots, 1}_{|S_0| \, \text{times}}, \underbrace{t_1,
    \ldots, t_1}_{|S_1|\, \text{times}}, \ldots, \underbrace{t_r,
    \ldots, t_r}_{|S_r| \,\text{times}})
  \]
  is one-to-one and the linear map $A$ induced by $\alpha$ is
  injective. Therefore,
  \[
  \cT Sec^1(\alpha(C)) = \cT(\alpha(Sec^1(C)))= A(\cT Sec^1(C))\;,\;
  \cT (\alpha(Z)) = A(\cT Z) \; , \; \cT(\alpha(C)) = A (\cT C), 
  \] and any fan structure in $\cT Sec^1(C)$ and $\cT Z$ translates
  immediately to a fan structure in $\cT Sec^1(\alpha(C))$ and $\cT
  \alpha(Z)$, by the injectivity of $A$. 
  From
  ~\eqref{eq:ST} we see that multiplicities match up, i.e.\
  $m_v = m_{\alpha(v)}$ for any regular points $v$, $\alpha(v)$. This
  concludes our proof.
\end{proof}

Our next goal is to explain the relationship between $\cT Sec^1(C)$
and the \emph{master graph} presented in
Section~\ref{sec:fundamental-graph}. First,
Proposition~\ref{pr:TropOfSecant} expresses the tropical secant
variety set-theoretically as the Minkowski sum of $\cT C$ and $\cT
Z$. Despite what one might think at first glance, a Minkowski sum of
fans does not have any canonical fan structure derived from those of
its summands. Some maximal cones can be subdivided, while others can
be merged into bigger cones.  Nonetheless, we can still use this
characterization to describe $\cT Sec^1(C)$ not just as a set, but as
a collection of four-dimensional weighted cones $\{\cT C + \sigma\}$
where $\sigma$ varies over maximal cones of $\cT Z$ whose sum with
$\cT C$ has dimension four. This presentation allows us to compute the
multiplicity of any regular point $\omega$ in $\cT Sec^1(C)$ as the
sum of weights of cones containing $\omega$, in agreement with the
spirit of Theorem~\ref{thm:GT} and Remark~\ref{rm:SumMultGT}.
 
Each maximal cone $\sigma$ in $\cT Z$ is represented by an edge in the
master graph. Thus, if we reduce the collection of four-dimensional
cones encoding $\cT Sec^1(C)$ by its lineality space $\cT C$ and
intersect it with the $(n-2)$-sphere, we obtain a subgraph of the
master graph. Intersections between cones of the collection $\cT
Sec^1(C)$ come in several flavors. If the intersection between a pair
of cones is four-dimensional, we call it an \emph{overlap}. If these
cones coincide, we say the overlap is \emph{complete}; otherwise, it
is \emph{partial}. If their intersection is three-dimensional, we call
it a \emph{crossing}. If they intersect at a common face of each, we
say that the crossing is \emph{nodal}; otherwise, it is
\emph{internal}. We wish to summarize all complete overlaps and nodal
crossings in our subgraph. This data is recorded in the {tropical
  secant graph} from Definition~\ref{def:tropSecantGraph}. For a
numerical example, see
Figure~\ref{fig:RanestadSecantAndGrobnerSecant}. Meanwhile, partial
overlaps and internal crossings are considered in
Theorems~\ref{thm:GrobnerTropSecGraphForN=6},~\ref{thm:GrobnerTropSecGraphForN=5}
and~\ref{thm:GrobnerTropSecGraphForN=4}, when discussing fan
structures. As we hinted in Theorem~\ref{thm:MasterGraphIsTropical}, a
special role is played by $\underline{b}=\{0, i_1, \ldots, i_{n-1}\}$
and $\underline{e}=\{i_1, \ldots, i_n\}$, the ``beginning'' and
``ending'' subsets, as Remark~\ref{rk:tropSecantGraphIdentifications}
shows.

\begin{definition}\label{def:tropSecantGraph}
  The \emph{tropical secant graph} is a weighted subgraph of the
  master graph in $\RR^{n+1}$, with nodes:
  \begin{enumerate}
  \item $D_{i_j}= {e_j}:=(0, \ldots, 0,
1, 0, \ldots,    0)$ 
    \qquad ($ 0\leq j\leq n$),
\item  $E_{i_j}= \sum_{k<j} i_k\cdot e_k + i_j\cdot (\sum_{k\geq j} e_k)=(0, i_1, \ldots, i_{j-1},
  {i_j}, \ldots, i_j)$ \qquad ($ 1\leq j\leq
  n-1$),
\item $F_{\underline{a}}=\sum_{i_j\in \underline{a}}{ e_j}$\quad where 
  $\underline{a}\subsetneq\{0, i_1, \ldots, i_n\}$ varies among all proper
  subsets containing at least two elements that are obtained from an arithmetic progression.
\end{enumerate}
The edges have positive weights:
\begin{enumerate}
\item $m_{ E_{i_j}, E_{i_{j+1}}} = \gcd(i_1, \ldots, i_j)\,
\gcd\limits_{j<t<n} (i_n-i_t)$ \quad  ($1\leq j \leq n-2$),
\item $m_{ D_{i_j},E_{i_j}}\!\! = \gcd\!\big(\gcd (i_1, \ldots, i_{j-1}) 
\! \! \gcd\limits_{j< l\leq n}\!\! (i_l-i_j)\, ;\!\! \gcd\limits_{0\leq k  <j}
\! \! (i_j-i_k) \,
  \! \gcd(i_{j+1}, \ldots, i_n)  \big)$\,   ($1\!\leq\! j\! \leq\!
  n-1$),
\item $m_{F_{\underline{a}},D_{i_j}} =
  \frac{1}{2}\sum\limits_{r}\varphi(r) \cdot\gcd\big(\gcd\limits_{i_l,
    i_k\notin \underline{a}} (\mid i_l-i_k\mid)\; ;
  \gcd\limits_{i_l, i_k\in\, \underline{a}\smallsetminus \{i_j\}} (\mid
  i_l-i_k\mid)\big)$ \quad where $\underline{a} = \{i_l\mid i_l \equiv i_j \!\!\!\pmod r\}$,
             $r\in \ZZ$ induces the subset $\aaa$ and $2\leq |\aaa| \leq n$.
\end{enumerate}
(By convention, a gcd over an empty set of indices is taken to be 0.)
\end{definition}

\begin{remark}
\label{rk:tropSecantGraphIdentifications}
We explain in words the transition from the master graph to the
tropical secant graph in reducing by $\cT C$. First, the edges
$D_{i_n}E_{i_{n-1}}$, $D_{i_n}h_{i_{n-1}}$ and $D_{i_0}h_{i_1}$ are
deleted. Second, the node $F_{0, i_1, \ldots, i_n}$ and all its
adjacent edges disappear. Third, the nodes
$h_{i_j}$ collapse to the corresponding nodes
$E_{i_j}$ for $1\leq j \leq n-1$. Lastly, if $F_{\underline{e}}$ (resp.\ $F_{\underline{b}}$) is a node in the master graph, we identify it with $E_{i_1}$ (resp.\ $E_{i_{n-1}}$) due to the equalities
  \begin{equation*}
     i_1\cdot F_{\underline{e}} = E_{i_1},\quad  (i_n-i_{n-1}) \cdot
     F_{\underline{b}} = E_{i_{n-1}} + (i_n-i_{n-1}) \cdot {\bf 1} +
     (-1) \cdot (0, i_1, \ldots, i_n).
  \end{equation*}
In this identification, the edges adjacent to the first node are added to those of
the second. We also merge the corresponding edges $E_{i_1}D_{i_1}$ and $F_{\underline{e}}D_{i_1}$ (resp.\ $E_{i_{n-1}}D_{i_{n-1}}$ and $F_{\underline{b}}D_{i_{n-1}}$) in the tropical secant graph, assigning the sum of their weights to the new edge. 

As in the case of the master graph, if we have a subset $\underline{a}
= \{i_j,i_k\}$ coming from an arithmetic progression, then the
corresponding node $F_{\underline{a}}$ is bivalent and can be removed
from the tropical secant graph. We then replace the two edges
$F_{\underline{a}}D_{i_j}$ and $F_{\underline{a}}D_{i_k}$ of equal
weight with a single edge $D_{i_j}D_{i_k}$ of the same weight. After
the above deletion of edges and collapse of nodes, some of the other
nodes may also become bivalent, so we are allowed to remove them as
well. However, we keep them 
to simplify notation.
\end{remark}

\medskip

By Theorem~\ref{thm:MainThm} the tropical secant graph characterizes
the tropicalization of the first secant of any monomial curve. Here is
a precse set-theoretic description of the associated graph:

\begin{corollary}
  The underlying graph of $\cT Sec^1(C)$ is obtained by
  gluing the graphs
  \begin{center}
 \begin{minipage}[c]{.5\linewidth}
 \includegraphics[scale=0.38]{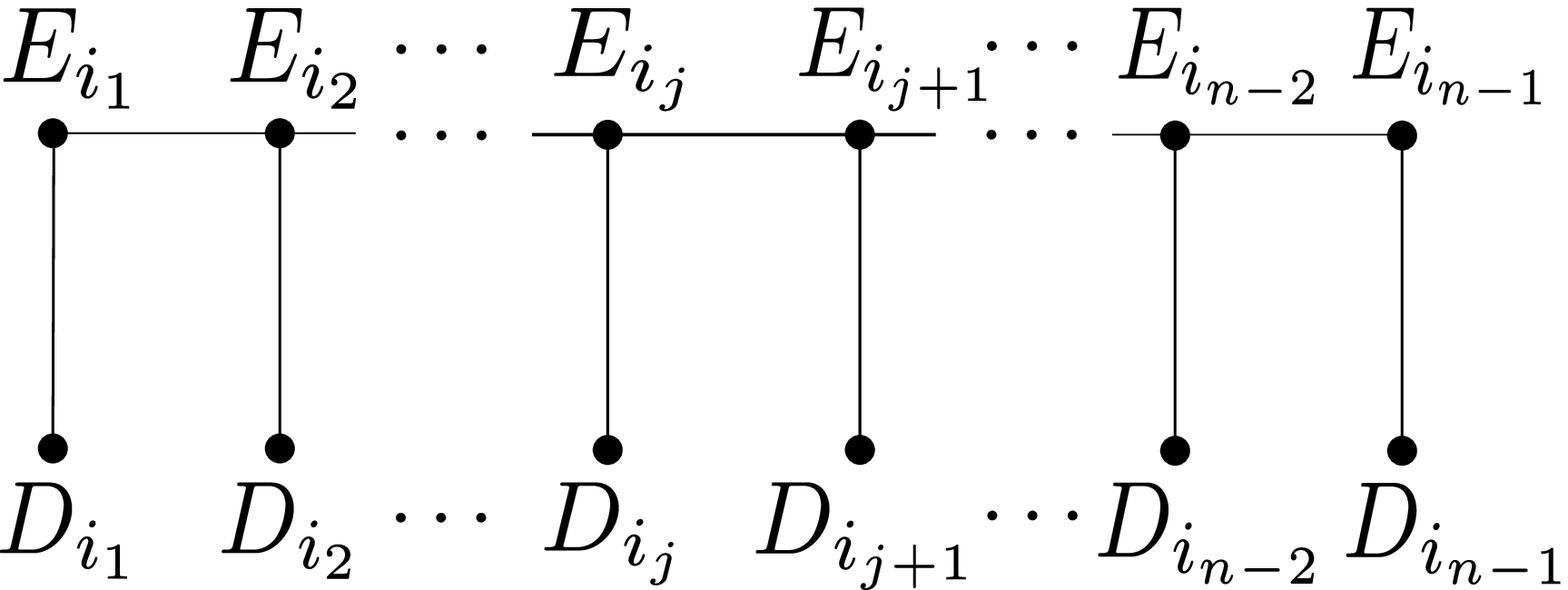} 
 \end{minipage}
 \begin{minipage}[c]{.45\linewidth}
  \centering \includegraphics[scale=0.2]{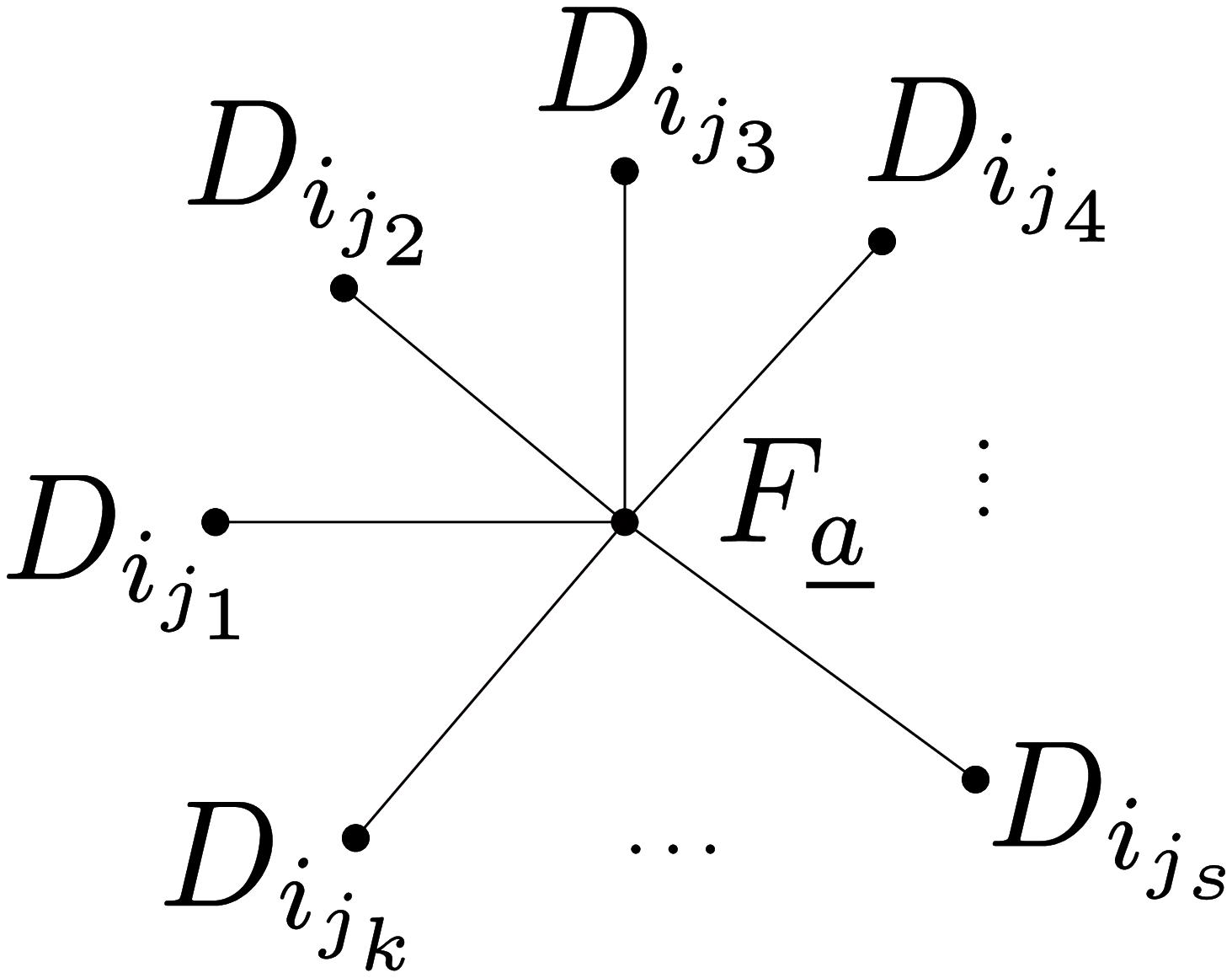}
\[ \small{\underline{a}
\neq \{0,
\ldots, i_n\}}\]
 \end{minipage}
  \end{center}
\noindent
 along all nodes $D_{i_j}$, and gluing together the
  nodes $E_{i_1}\equiv F_{i_1, \ldots, i_n}$, $E_{i_{n-1}}\equiv F_{0,
    \ldots, i_{n-1}}$.
\end{corollary}
The remainder of this section is devoted to proving
Theorem~\ref{thm:MainThm}, and in particular, to explaining the
mysterious formulas for the weights of the tropical secant
graph. We obtain these numbers using
formula~\eqref{eq:ST}. 
The next propositions and lemmas characterize each one of the
quantities involved in~\eqref{eq:ST}. Our proofs use similar
techniques to the ones of~\cite[Lemma 4.4]{CoxSidman07}.
\begin{proposition}\label{pr:sizeOfFiber}
  Let $\alpha\colon \TP^{2n+2}\to \TP^{n+1}$ be the Hadamard monomial map
  associated to the matrix $(I_{n+1}\mid I_{n+1})\in \ZZ^{(n+1)\times
    2(n+1)}$ and $C, Z$ as in Proposition~\ref{pr:ReparamSecant}. Then, the generic
  fiber of $\alpha_{|_{C\times Z}}$ has size $2$, giving $\delta= 2$ in
  formula~\eqref{eq:ST}.
\end{proposition}

\begin{proof}
  Generically, by equation~\eqref{eq:2}, the elements of the fiber
  of $\alpha$ at
  a point $p$ are in one-to-one correspondence with pairs of points in the
  curve $C$ that are collinear with $p$.  By switching the role of
  these two points in the secant map, we know that the generic fiber
  of $\alpha_{|_{C\times Z}}$ has size at least two. 
  Lemma~\ref{lm:No4CoplanarInCGenerically} implies that 
  it has
  exactly two points.
\end{proof}

\begin{lemma}\label{lm:No4CoplanarInCGenerically}
  For almost all points $p$ in the secant variety of $C$, $p$ lies on a
  single 
  one secant line which, in addition, intersects the curve $C$ at exactly
  two points.
\end{lemma}
\begin{proof}
  We restrict the secant map $\phi$ to the open torus $\TP^3$ mapping
  $(a,s,t) \mapsto (a s^{i_k}+ (1-a)t^{i_k} )_{ 0\leq k \leq n}$. We
  claim that it suffices to prove the lemma for points in the image of
  ${\phi}$, when $n=4$ and the exponents are coprime and distinct

  Assume the statement is true for $n=4$ and coprime
  exponents. Consider all maps $\phi_j$ obtained by composing the map
  $\phi$ with the projections $\pi_j$ onto the five coordinates $0,
  1,2, 3, j$ ($4\leq j \leq n$).  Let $d_j={\gcd(i_1, i_2, i_3, i_j)}$
  and reparameterize $\phi_j$ using the identities $x:=s^{d_j}$ and
  $y:=t^{d_j }$, that is, define $\tilde{\phi_j}\colon \TP^3 \to
  \CC^{n+1}$ as
   \[
 \tilde{\phi_j}(a, x,y):=(1, ax^{\frac{i_1}{d_j}}+(1-a)y^{\frac{i_1}{d_j}},
 ax^{\frac{i_2}{d_j}}+(1-a)y^{\frac{i_2}{d_j}}, ax^{\frac{i_3}{d_j}}+(1-a)y^{\frac{i_3}{d_j}},
 a x^{\frac{i_j}{d_j}}+(1-a) y^{\frac{i_j}{d_j}}).
 \]
 Since the exponents of $\tilde{\phi}_j$ are coprime and the lemma
 holds for $n=4$ by assumption, we know that the fiber of
 $\tilde{\phi}$ over a point
 $\tilde{\phi}_j(a,x,y)$ contains only two points, namely the points
 $(a,x,y)$ and $(1-a, y, x)$. Therefore, any two points $(a,s,t)$ and
 $(a',s',t')$ in the fiber of $\phi$ over the same point satisfies
 $a=a'$, $(s')^{d_j} =s^{d_j}$ and $(t')^{d_j}=t^{d_j}$ for all $4\leq j
 \leq n$, up to symmetry. Since $\gcd(d_4,\ldots, d_n)=1$ we conclude $s=s'$ and
 $t=t'$.

 We now treat the case where $n=4$ and the exponents are coprime and
 pairwise distinct.  To prove our result, it suffices to show that the
 Zariski closure of the set of points in $Sec^1(C)\smallsetminus C$
 which are intersections of two distinct secant lines of $C$ has
 dimension at most two. We parameterize these points by tuples
 $(s,t,u,v)$ of distinct complex numbers, corresponding to four
 coplanar, non-collinear points in $C$.  It suffices to show that the
 set $W$ of such tuples has dimension at most two.

The variety $W$ is
cut out by all $3\times 3$-minors of the $3\times 4$-matrix with rows
$(t^{i_j}-s^{i_j})_{1\leq j\leq 4}$, $(u^{i_j}-s^{i_j})_{1\leq j\leq
  4}$ and $(v^{i_j}-s^{i_j})_{1\leq j\leq 4}$.  Note that for all
minors to vanish, it is
enough to show that two of them do. We pick the ones corresponding
to columns $\{1,2,3\}$ and $\{1,2,4\}$. These minors are precisely the
determinants of the $4\times 4$ generalized Vandermonde matrices
\[
M_{i_1,i_2,i_3}:=\left(
\begin{array}{cccc}
  1 & s^{i_1} & s^{i_2} & s^{i_3}\\
1& t^{i_1} & t^{i_2} & t^{i_3}\\
  1 & u^{i_1} & u^{i_2} & u^{i_3}\\
1& v^{i_1} & v^{i_2} & v^{i_3}
\end{array}
\right) \quad , \quad
M_{i_1,i_2,i_4}:=\left(
\begin{array}{cccc}
  1 & s^{i_1} & s^{i_2} & s^{i_4}\\
1& t^{i_1} & t^{i_2} & t^{i_4}\\
  1 & u^{i_1} & u^{i_2} & u^{i_4}\\
1& v^{i_1} & v^{i_2} & v^{i_4}
\end{array}
\right).
\]
Because $s,t,u$ and $v$ are all distinct, we can divide the
determinant of these two matrices by the product of pairwise
differences among our four variables, that is, by the \emph{Vandermonde
  determinant} $V(s,t,u,v)$. The resulting polynomials are the
\emph{Schur polynomials} $S_{i_1, i_2, i_3}$ and $S_{i_1, i_2, i_4}\in
\ZZ[s,t,u,v]$.

Let $g=\gcd(i_1, i_2, i_3)$ and $h=\gcd(i_1, i_2, i_4)$. Note that $\gcd(g,h)=1$. By \cite[Theorem 3.1]{SchurFunctions} we
can factorize the previous Schur polynomials over $\ZZ[s,t,u,v]$ as
\begin{equation*}
  \begin{aligned}
    S_{i_1, i_2, i_3} & = V(s^g,t^g,u^g,v^g)/V(s,t,u,v) \cdot T_{i_1, i_2,
      i_3}(s,t,u,v),\\
    S_{i_1, i_2, i_4} & = V(s^h,t^h,u^h,v^h)/V(s,t,u,v) \cdot T_{i_1, i_2,
      i_4}(s,t,u,v),
  \end{aligned}
\end{equation*}
where $T_{i_1, i_2, i_3}$ and $T_{i_1, i_2, i_4}$ are either constants
or irreducible over $\CC[s,t,u,v]$. These two polynomials are
homogeneous of total degree $g(i_3/g+i_2/g+i_1/g-3)$ and
$h(i_4/h+i_2/h+i_1/h-3)$, and of degree $i_3-2g$ and $i_4-2h$ in each
variable $s,t,u$ and $v$. By comparing their multidegrees, we conclude that
the polynomials $T_{i_1,i_2,i_3}$ and $T_{i_1,i_2, i_4}$ are coprime.
Next, we claim that the polynomials $V(s^g,t^g,u^g,v^g)/V(s,t,u,v)$
and $V(s^h,t^h,u^h,v^h)/V(s,t,u,v)$ are coprime. This follows from the
well-known identity $(x^g-y^g)=\prod_{\zeta^g=1}(x-\zeta
y)$.

From the previous two observations, we see that if $S_{i_1, i_2, i_3}$
and $S_{i_1, i_2, i_4}$ have a common factor, then either $V(s^h,t^h,
u^h, v^h)/V(s,t,u,v)$ and $T_{i_1, i_2, i_3}$ or the polynomials
$V(s^g,t^g, u^g, v^g)/V(s,t,u,v)$ and $T_{i_1, i_2, i_4}$ would have a
common factor over $\CC[s,t,u,v]$. Without loss of generality, assume
the first pair of polynomials is not coprime. By irreducibility of
$T_{i_1, i_2, i_3}$ and the factorization of $V(s^h,t^h,u^h,v^h)$ into
linear factors involving only two of the variables, this forces
$T_{i_1, i_2, i_3}$ to be linear and to involve only two variables,
contradicting the degree formulas provided above. Hence, we conclude
that $S_{i_1, i_2, i_3}$ and $S_{i_1,i_2,i_4}$ are coprime.  In
particular, this says that the hypersurfaces in $\CC^4$ defined by
$S_{i_1, i_2, i_3}$ and $S_{i_1, i_2, i_4}$ have distinct reduced
irreducible components of dimension three. Therefore, $\dim W\leq 2$
and this ends our proof.
\end{proof}

Next, we compute all cones of the form $\cT C + \RR_{\geq
  0}\otimes\sigma$ of dimension at most three, where $\sigma$ runs
over edges of the master graph $\cT Z$. We discard these cones from
the $\cT \Sec^1(C)$. In addition, we consider all possible pairs
$\sigma, \sigma'$ of such maximal cones in $\cT Z$ to find all
pairwise intersections $(\cT C + \RR_{\geq 0}\otimes \sigma) \cap (\cT
C + \RR_{\geq 0}\otimes\sigma')$ which are complete overlaps or nodal
crossings. By an elementary exhaustive case by case analysis, we
conclude:
\begin{lemma}\label{lm:transitionFromMasterToTropicalSecantGraph}
  After reducing the master graph by the linear space $\cT C$, the
  only non-maximal cones, complete overlaps and nodal crossings are as
  follows:
 \begin{enumerate}
 \item The cones $ \cT C +\RR_{\geq 0}\langle D_0, h_{i_1}\rangle$, 
   $\cT C +\RR_{\geq 0}\langle D_{i_n}, E_{i_{n-1}}\rangle $, $\cT C +\RR_{\geq 0}\langle D_{i_n},
   h_{i_{n-1}}\rangle $ and $
   \cT C +\RR_{\geq 0}\langle F_{0,
     i_1, \ldots, i_n}, D_{i_j}\rangle $ ($0\leq j \leq n$) are not maximal, so we disregard them.
\item The node $F_{\{0, i_1, \ldots, i_n\}}={\bf 1} \in \cT C$, so we
  eliminate it from the graph, together with all its $n+1$ adjacent edges.
  \item For all $1\leq j \leq n-2$, we have equalities $\cT C + \RR_{\geq 0}\langle
    E_{i_j}, D_{i_j}\rangle = \cT C +\RR_{\geq 0}\langle h_{i_j}, D_{i_j}
    \rangle $ and $\cT C +\RR_{\geq 0}\langle E_{i_j}, E_{i_{j+1}}\rangle
    = \cT C +\RR_{\geq 0}\langle h_{i_j}, h_{i_{j+1}} \rangle $ because
    $E_{i_j} \equiv h_{i_j}$ modulo $\cT C$. Hence, we disregard all
    nodes $h_{i_j}$ and their adjacent edges.
  \item $i_1\cdot F_{\underline{e}} = E_{i_1}$ and $(i_n-i_{n-1})
    \cdot F_{\underline{b}} \equiv E_{i_{n-1}}$ modulo $\cT C$, where
    $\underline{e}=\{i_1, \ldots, i_n\}$ and $\underline{b}=\{0, i_1,
    \ldots, i_{n-1}\}$. Thus, the maximal cones $\cT C +\RR_{\geq
      0}\langle F_{\underline{e}}, D_{i_1}\rangle $ and
    $\cT C +\RR_{\geq 0}\langle E_{i_1}, D_{i_1}\rangle $ coincide, as
    well as $\cT C +\RR_{\geq 0}\langle F_{\underline{b}}, D_{i_{n-1}}\rangle
    $ and $\cT C +\RR_{\geq 0}\langle E_{i_{n-1}}, D_{i_{n-1}}\rangle
    $.
 \end{enumerate}
\end{lemma}

\begin{proof}[\textbf{Proof of Theorem~\ref{thm:MainThm}.}]
  Proposition~\ref{pr:TropOfSecant} and
  Lemma~\ref{lm:transitionFromMasterToTropicalSecantGraph} prove
  that the cone from $\cT C$ over the tropical secant graph coincides
  with the tropical variety $\cT Sec^1(C)$ as a
  collection of four-dimensional weighted cones. In particular, this shows
  that the tropical secant graph combines all nodes and edges coming
  from nodal crossings and complete overlaps. By
  formula~\eqref{eq:ST}, the multiplicity at a regular point $\omega$
  of $\cT Sec^1(C)$ is the sum of all weights of four-dimensional
  cones $\cT C + \sigma$ containing $\omega$, where $\sigma$ is a
  maximal two-dimensional cone of $\cT Z$. Furthermore, if $m_\sigma$
  is the weight of $\sigma$ in $\cT Z$, the formula weights  $\cT C
  + \sigma$ by 
  \begin{equation*}
\frac{1}{2} \cdot m_\sigma \cdot { index
  }((\mathbb{L}_{\sigma}+\cT C) \cap \ZZ^{n+1} , (\mathbb{L}_\sigma \cap
  \ZZ^n)+(\cT C \cap \ZZ^{n+1})).\label{eq:8bis}
  \end{equation*}

We now prove that this number yields the weight of the corresponding edge in
the tropical secant graph, after combining weights in complete
overlaps following
Remark~\ref{rk:tropSecantGraphIdentifications}. Suppose the cone
$\sigma$ is generated by integer vectors $\mathbf{x}, \mathbf{y} \in \ZZ^{n+1}$. Let
${\bf l_1} =\bf{1}$ and ${\bf l_2} = (0, i_1, \ldots, i_n)$ be the
generators of the primitive lattice $\Lambda$ in $\cT C\cap
\ZZ^{n+1}$. The lattice index in the above formula is the gcd of all
$4\times 4$-minors of the matrix $(\mathbf{x}\mid \mathbf{y} \mid {\bf l_1}\mid {\bf l_2})$
divided by the gcd of all $2\times 2$-minors of the matrix $(\mathbf{x}\mid \mathbf{y})$. These
gcd's are computed as the product of the nonzero diagonal elements of
the Smith normal form of each matrix.

As an example, we show how to obtain the multiplicity $m_{D_{i_j},
  E_{i_j}}$ ($2\leq j \leq n-2$). The remaining multiplicities can be
computed analogously. The edge $D_{i_j}E_{i_j}$ is associated to
precisely two edges in the master graph giving two four-dimensional
cones in $\cT Sec^1(C)$ that overlap completely. These two edges are
$\sigma=D_{i_j}E_{i_j}$ and
$\sigma'=D_{i_j}h_{i_j}$. 
  From Definition 2.1, the multiplicity $m_{\sigma}$ equals $\gcd(i_1,
  \ldots, i_j)$, which is also the gcd of the $2\times 2$-minors of the matrix
  $(D_{i_j}|E_{i_j})$. Likewise, $m_{\sigma'} = \gcd(i_j, \ldots,
  i_n)$ is the gcd of the $2\times 2$-minors of $(D_{i_j}|h_{i_j})$. These
  numbers are precisely the denominators in the formulas for computing
  the indices associated to $\sigma$ and $\sigma'$ in~\eqref{eq:8}. 
Since $E_{i_j} \equiv h_{i_j}$ mod $\Lambda$, we conclude
\begin{align*}
  m_{D_{i_j},E_{i_j}} &= \frac{1}{2}\big(\gcd(4\times 4\text{-minors of }(D_{i_j}|E_{i_j}|{\bf
    l_1}|{\bf l_2})) + \gcd(4\times 4\text{-minors of
  }(D_{i_j}|h_{i_j}|{\bf l_1}|{\bf l_2}))\big) \\ & =
  \gcd(4\times 4\text{-minors of }(D_{i_j}|E_{i_j}|{\bf
     l_1}|{\bf l_2})). 
\end{align*}

We now compute the gcd of all $4\times 4$-minors of the matrix
$(D_{i_j}|E_{i_j}|{\bf l_1}|{\bf l_2})$. For simplicity, we work with
the transpose of this matrix. By elementary operations between rows
that do not change the minors, we alter the second, third and fourth
rows, and expand the minors along the first row, reducing our problem
to computing the $3\times 3$-minors of the matrix
\[
\left(\begin{array}{cccc|ccc}
0 & i_1 &  \ldots& i_{j-1} & i_j & \ldots & i_j\\
1 &   1 &\ldots& 1 & 1& \ldots &1\\
0&  0 &\ldots& 0   &i_{j+1}-i_j &\ldots& i_n-i_j
\end{array}\right)\in \ZZ^{3\times n},
\]
where $i_0=0$.
All non-vanishing $3\times 3$-minors must involve columns from the two
constituent blocks.  The gcd of the minors involving two columns of
the left-hand side is $\gcd (i_1, \ldots, i_{j-1}) \gcd_{j< l\leq   n} (i_l-i_j)$, whereas the gcd of the minors involving two columns
of the rightmost block equals the product $\gcd(i_1,\ldots, i_j)
\gcd_{j<l<n}(i_n-i_l)$. This justifies the formula for
$m_{D_{i_j},E_{i_j}}$ in
Definition~\ref{def:tropSecantGraph}.
\end{proof}


We now study partial overlaps and internal crossings among cones from
$\cT C$ over edges of the tropical secant graph. An
exhaustive case by case analysis shows that if $n\geq 5$, there are no
partial overlaps, and that if $n \geq 6$, there are no internal crossings as
well. For the case
$n=4$, Lemma~\ref{lm:n4Overlaps} and Theorem~\ref{thm:GrobnerTropSecGraphForN=4} indicate that both partial
overlaps and internal crossings are possible. 

Partial overlaps and internal crossings prevent us from inferring a
fan structure for $\cT Sec^1(C)$ from the tropical secant
graph. However, we may introduce new nodes at crossings,
subdivide edges while preserving their weights or merge overlapping
edges and their weights to create a new graph. If this surgery is performed appropriately,
the new graph  encodes the fan structure of our tropical variety as a subfan of the
\emph{Gr\"obner fan} of the homogeneous ideal defining the secant variety. This
motivates the following definition:

  \begin{definition}
    A \emph{Gr\"obner tropical secant graph} for a projective monomial
    curve $C$ parameterized by $n$ coprime distinct integers is a
    weighted graph in $\RR^{n+1}$ whose cone from $\cT C$ gives the
    weighted Gr\"obner fan structure on $\cT Sec^1(C)$.
  \end{definition}

  Unsurprisingly, the complexity of the surgery required to transform
  the tropical secant graph into a Gr\"obner tropical secant graph
  depends on the value of $n$. We present the results for $n>4$ and
  postpone the discussion of the case $n=4$ until the next section.
\begin{theorem}\label{thm:GrobnerTropSecGraphForN=6}
  The tropical secant graph of a monomial curve in $\pr^n$ is a Gr\"obner tropical secant graph for $n\geq 6$.
\end{theorem}
\begin{proof}
  The proof of this result is elementary and it boils down to
  analyzing intersections between cones from $\cT C$ over pairs of edges in the tropical secant
  graph. If $n\geq 6$, we do not get any partial overlaps or internal crossings.
\end{proof}

\begin{theorem}\label{thm:GrobnerTropSecGraphForN=5}
  For a monomial curve in $\pr^5$, a Gr\"obner tropical secant graph
  may be constructed from the tropical secant graph by adding finitely
  many nodes and subdividing edges accordingly. More precisely, we
  must add nodes $P_{\aaa,j, \ap,k} \in (\cT C + \langle
  F_{\aaa},D_{i_j}\rangle)\bigcap (\cT C + \langle
  F_{\ap},D_{i_k}\rangle) $ where $F_{\aaa}, F_{\ap}$ are nodes in the
  master graph, and the set $\{\aaa,j, \ap,k\}$ together with the index set
  $\{i_1, \ldots, i_5\}$ satisfy one of the following three
  conditions:
  \begin{enumerate}
  \item 
$j=5, k=0$, $i_4+i_1=i_2+i_3$ and the
subsets $\aaa, \ap$ are either  $\aaa=\{i_3,i_4,i_5\},  \ap=\{0,i_1,i_3\}$
or  $\aaa=\{i_2,i_4,i_5\},  \ap=\{0,i_1,i_2\}$;
\item $\aaa=\{i_j,i_r,i_k,i_u\},
  \ap=\{i_u,i_k,i_t\}$, $i_l \notin \aaa\cup \ap$,  $i_r+i_t=i_l+i_u$
  and either
$j>r>l>t, r>u>t$ and $u>k$, or $j<r<l<t, r<u<t$ and
$u<k$;
\item $\aaa=\{i_j,i_u,i_k,i_r\}, \ap=\{i_j,i_u,i_k,i_t\}$, $i_l \notin
  \aaa\cup \ap$ and
  $i_r+i_t=i_l+i_u$,  while $j>u>k, r>l>t$ and $r>u>t$.
  \end{enumerate}

  For each new node $P_{\aaa,j, \ap,k}$, we subdivide the edges
  $F_{\aaa}D_{i_j}$ and $F_{\ap}D_{i_k}$ of the tropical secant graph
  to get edges $F_{\aaa}P_{\aaa,j, \ap,k}$, $P_{\aaa,j,
    \ap,k}D_{i_j}$, and $F_{\ap}P_{\aaa,j, \ap,k}$,
  $P_{\aaa,j, \ap,k}D_{i_k}$, preserving the
  original weights.
\end{theorem}
\begin{proof}
  The proof is tedious, yet elementary. As before, we consider all
  intersections between cones from $\cT C$ over edges of the tropical
  secant graph.  While there are no partial overlaps, we have three
  types of internal crossings. We write down each intersection point:
  \begin{enumerate}
  \item Suppose $|\aaa|= |\aaa'|=3$, say  $\aaa=\{i_j,i_r,i_u\},
  \ap=\{i_u,i_k,i_t\}$. Let $i_l\notin \aaa\cup \aaa'$.
   Assume the cones over the edges $F_{\aaa}D_{i_j}$
 and  $ F_{\ap}D_{i_k}$ intersect and $j>k$. Then, $j=5, r=4, t=1,
k=0$, $u=2$ or $3$, 
  $i_4+i_1=i_2+i_3$ and
\[
(i_r-i_l)\cdot F_{\aaa} +(i_j-i_r) \cdot D_{i_j} =
(i_r-i_u)\cdot  F_{\ap} + (i_t-i_k) \cdot D_{i_k} + (-i_l)
u\cdot {\bf 1} + (0, i_1, \ldots, i_5).
\]
\item Suppose $|\aaa|=4, |\aaa'|=3$, say $\aaa=\{i_j,i_r,i_k,i_u\},
  \ap=\{i_k,i_u,i_t\}$. Let $i_l\notin \aaa\cup \aaa'$ 
  and assume the cones over $F_{\aaa}D_{i_j}$ and 
   $ F_{\ap}D_{i_k}$ intersect. Then,
  $i_r+i_t=i_l+i_u$ and
\[
(i_r-i_l)\cdot F_{\aaa} +(i_j-i_r) \cdot D_{i_j} =
(i_l-i_t)\cdot  F_{\ap} + (i_u-i_k) \cdot D_{i_k} + (-i_l)
\cdot {\bf 1} + (0, i_1, \ldots, i_5).
\]
In this case, all coefficients (except for $-i_l$ and $1$) have the
same sign, which can be negative. If the latter occurs, we multiply
the previous identity by $-1$, obtaining the internal crossing of the
two cones. This expression gives us up to ten extra points, determined by the
inequalities $j>r>l>t, r>u>t$ and $u>k$, or $j<r<l<t, r<u<t$ and
$u<k$.
\item Suppose $|\aaa|=|\aaa'|=4$, say  $\aaa=\{i_j,i_u,i_k,i_r\},
  \ap=\{i_j,i_u,i_k,i_t\}$, and assume $j>k$. Let $i_l\notin \aaa\cup
  \aaa'$. If the cones
over  $F_{\aaa}D_{i_j}$ and $ F_{\ap}D_{i_k}$
  intersect, then $i_r+i_t=i_l+i_u$ and
\[
(i_r-i_l)\cdot F_{\aaa} +(i_j-i_u) \cdot D_{i_j} =
(i_l-i_t)\cdot  F_{\ap} + (i_u-i_k) \cdot D_{i_k} + (-i_l)
\cdot {\bf 1} + (0, i_1, \ldots, i_5).
\]
By requiring all mandatory coefficients to be positive, we obtain 
$j>u>k, r>l>t$ and $r>u>t$. This gives twelve possibilities for such
internal crossings. 
  \end{enumerate}\vspace{-3.2ex}

\end{proof}


\section{The Newton polytope of the secant hypersurface in $\pr^4$}
\label{sec:newt-polyt-secant}

In this section, we focus our attention on monomial curves in
$\pr^4$. In this situation, the first secant variety becomes a
hypersurface and we wish to obtain its definition homogeneous equation
from the tropical secant graph. A first step
towards a complete solution would be to compute the \emph{Newton
  polytope} of the defining equation $f$, i.e.\ the convex hull of all
exponent vectors such that the corresponding monomial appears with a
nonzero coefficient in $f$.

 We start by summarizing the
methods known to address this question. Secondly, we illustrate this
technique with an example appearing in the literature
\citep{SecantMonCurves}. We conclude the section by constructing the
Gr\"obner tropical secant graph of any monomial curve in $\pr^4$. The
latter is unnecessary for computing the associated Newton
polytope, but it allows us to predict some of its rich combinatorics.

We first explain the connection between $\cT(f)$ and $\NP(f)$ for an
irreducible polynomial $f$ in $n+1$ variables defined over $\CC$. For
a vector $w \in \RR^{n+1}$, the initial form $\init_w(f)$ is a
monomial if and only if $w$ is in the interior of a chamber of the
inner normal fan of $\NP(f)$.  The tropicalization of the hypersurface
$(f=0)$ is the union of all the codimension one cones in the normal fan of
$\NP(f)$.  The multiplicity of a maximal cone in $\cT (f)$ is the
lattice length of the edge of $\NP(f)$ normal to that cone.

A construction of the Newton polytope $\NP(f)$ from its
\emph{weighted} normal fan $\cT (f)$ was developed in
\cite{TropDiscr}, and it is known as the \emph{ray-shooting algorithm}.
We describe it in Theorem~\ref{thm:RayShoot} below:
\begin{theorem} \label{thm:RayShoot}
  Suppose $w \in \RR^{n+1}$ is a generic vector so that the ray $(w +
  \RR_{>0}\, e_i)$ intersects $\cT(f)$ only at regular points, for all
  $i$.  Let $\cP^w$ be the vertex of the Newton polytope 
  $\cP$ 
of $f$ that attains the maximum of $\{w \cdot x : x \in \cP\}$.
  Then, the $i^\text{th}$ coordinate of $\cP^w$ equals
  \begin{equation*}
\sum_{v} m_v \cdot |l^v_i|, 
\end{equation*}
where the sum is taken over all points $v \in \cT(f) \cap (w + \RR_{>
  0} e_i)$, $m_v$ is the multiplicity of $v$ in $\cT(f)$, and
$l^v_i$ is the $i^\text{th}$ coordinate of the primitive integral
normal vector $l^v$ to the maximal cone in $\cT(f)$ containing $v$.
\end{theorem}

This theorem allows us to compute the vertices of $\NP(f)$ when
$\NP(f)$ lies in the positive orthant and touches all coordinate
hyperplanes, i.e.\ when $f$ is not divisible by any non-constant
monomial.  Note that we do not need a fan structure on $\cT(f)$ to use
Theorem~\ref{thm:RayShoot}.  A description of $\cT(f)$ as a weighted
set provides enough information to compute vertices of $\NP(f)$ in any
generic direction. Obtaining a single vertex using
Theorem~\ref{thm:RayShoot} gives us the multidegree of $f$ with
respect to the grading given by the intrinsic lattice $\Lambda$.

The entire polytope $\NP(f)$ can be computed by iterating the
ray-shooting algorithm with different objective vectors (one per
chamber). A method to choose these vectors appropriately was developed
in \citep[Algorithm 2]{Mega09}: the \emph{walking algorithm}. The core of
the method is to keep track of the cones that we meet while
ray-shooting from a given objective vector, and use the list of such
cones to walk from chamber to chamber in the normal fan of
$\NP(f)$. Along the way, we pick objective vectors inside each
chamber, and we repeat the shooting algorithm.  We illustrate these
methods with an example:
\begin{example} \label{ex:Ranestad} The first secant variety of the
  monomial curve $t\mapsto (1:t^{30}: t^{45}: t^{55}:t^{78})$ in
  $\pr^4$ is known to be a hypersurface of degree 1820 \citep[Example
  3.3]{SecantMonCurves}. Here, we compute the tropical secant graph of
  the set $\{30, 45, 55, 78\}$. Using this data as input for
  the ray-shooting and walking algorithms, we calculate the Newton
  polytope of the secant threefold.

By Theorem~\ref{thm:MainThm}, we encode the tropical hypersurface
$\cT Sec^1(C)\subset \RR^5$ as a 
graph  in $\RR^3$ depicted in the left of Figure~\ref{fig:RanestadSecantAndGrobnerSecant}.
\begin{figure}
  \begin{minipage}[t]{.46\linewidth}
    \includegraphics[scale=0.43]{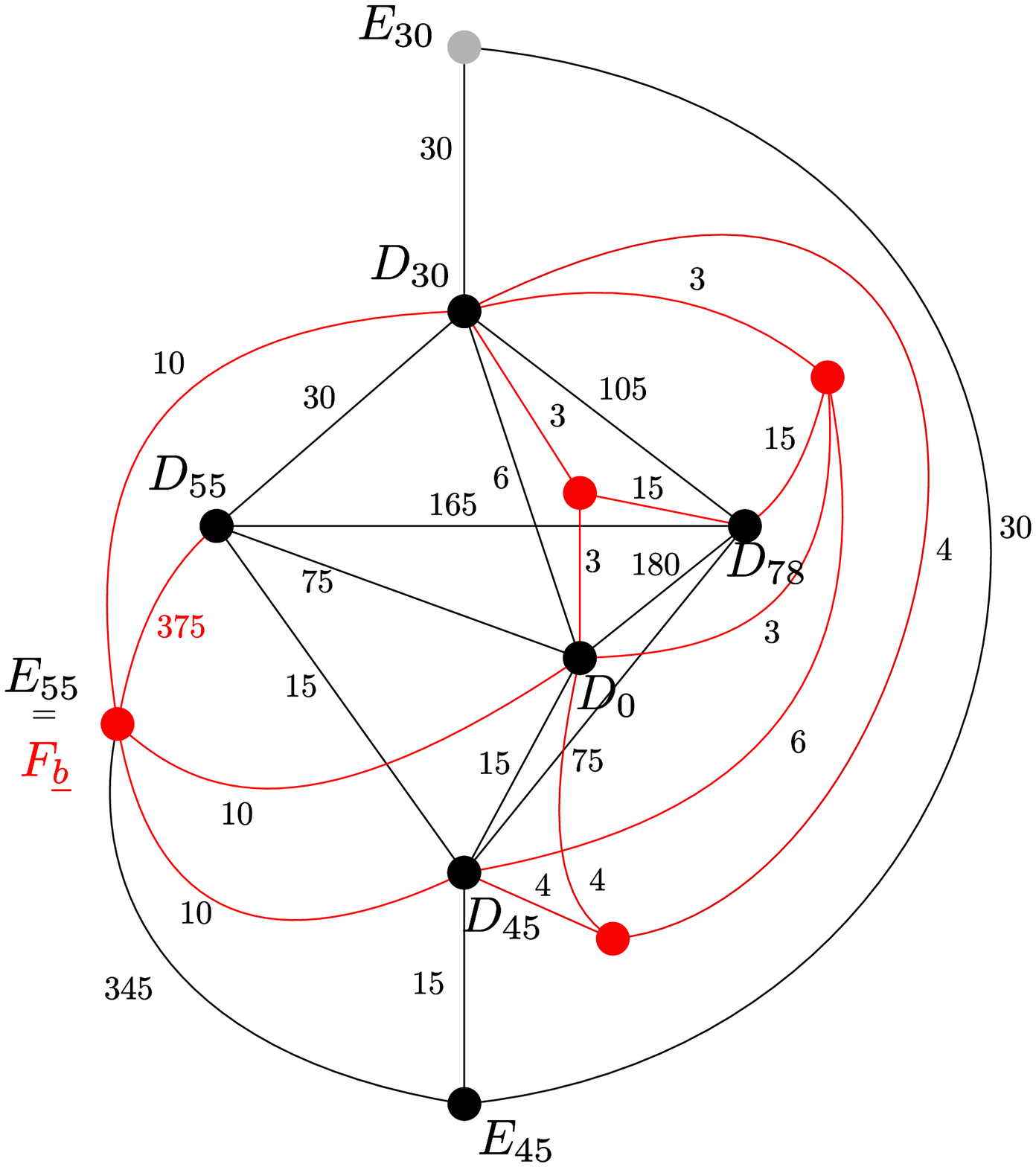}
  \end{minipage}\hfill
  \begin{minipage}[t]{.52\linewidth}
    \includegraphics[scale=0.43]{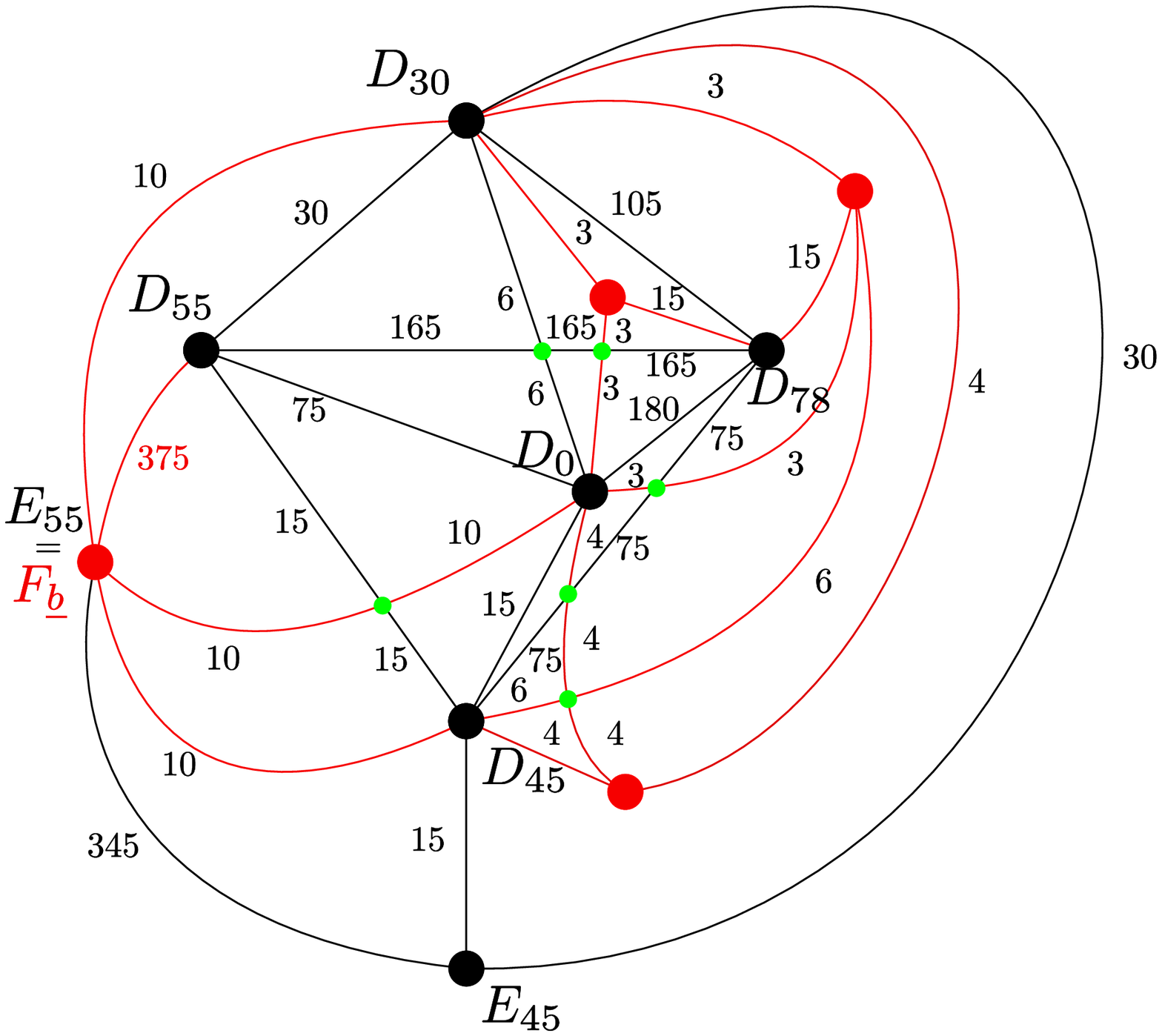}
  \end{minipage}
\caption{The tropical secant graph and the Gr\"obner tropical secant
  graph of the monomial curve $(1:t^{30}:t^{45}:t^{55}:t^{78})$ in
  $\pr^4$.}
  \label{fig:RanestadSecantAndGrobnerSecant}
\end{figure}
The eleven nodes in the graph have coordinates $D_0=e_0$,
$D_{30}=e_1$, $D_{45}=e_2$, $D_{55}=e_3$, $D_{78}=e_4$, $E_{30}= (0,
30, 30, 30, 30)$, $E_{45}=(0,30,45,45,45) $, $F_{0,30,45,55}\equiv
E_{55}=(0,30, 45, 55, 55)$, $F_{0, 30, 45}=(1,1,1,0,0)$,
$F_{0,30,78}=(1,1,0,0,1)$, and $F_{0, 30, 45, 78}=(1,1,1,0,1)$.  The
unlabeled red nodes in the
picture 
indicate nodes of type $F_{\aaa}$, where the subset $\aaa$ consists of
the indices of nodes $D_{i_j}$ adjacent to it.  Note that, in this
example, the nodes $E_{55}$ and $F_{0,30, 45, 55}$ are identified
modulo the lineality space, as predicted by
Definition~\ref{def:tropSecantGraph}. In particular, the edges $E_{55}D_{55}$ and $F_{0,30,45,55}D_{55}$ of
the master graph  coincide in
the tropical secant graph and the old weights add up to $375=345+30$,
as Figure~\ref{fig:RanestadSecantAndGrobnerSecant} shows.
After removing the bivalent gray node $E_{30}$, we have a graph with
$10$ nodes and $23$ edges.

Finally, we apply the ray-shooting and walking algorithms to recover
the Newton polytope of its defining equation.  From our computation,
we see that its multidegree with respect to the lattice $\Lambda
=\ZZ\langle {\bf{1}}, (0, 30, 45, 55, 78)\rangle$ is $(1\,820, 76\,950
)$, recovering the degree value from~\cite{SecantMonCurves}.  The
polytope has 24 vertices and $f$-vector $(24, 38, 16)$. The difference
between the number of facets of the polytope and the number of nodes
in the tropical secant graph shows that this graph does not reflect
the Gr\"obner fan structure of the tropical variety. In particular, we
are missing six vertices which correspond to internal crossings of the
graph. In the right of
Figure~\ref{fig:RanestadSecantAndGrobnerSecant}, we indicate these
six missing vertices with small green nodes.  After adding them to the
picture we obtain the Gr\"obner tropical secant graph of the curve
$C$, which is a planar graph in $\Sn^2$ with 16 nodes and 38 edges.
Each edge emanating from a new small green node corresponding to an internal
crossing inherits the weight of the original edge in the tropical
secant graph.  The complement of this graph has 24 connected
components, which matches the number of vertices of our polytope.
Using \texttt{LaTTe}, we see that the polytope contains $7\,566\,849$
lattice points, which gives an upper bound for the number of monomials
in its defining equation.
\end{example}

We conclude the section by building the Gr\"obner tropical secant
graph of any monomial curve in $\pr^4$. From our previous example, we
already know that the tropical secant graph can be non-planar. We
provide a theorem that indicates which internal crossings need to be
added to make it planar. However, these are not the only possible
intersections: we can have partial overlaps between edges. Luckily,
there are only three types of partial overlaps. We describe them in
the next lemma, which follows notation from
Definition~\ref{def:tropSecantGraph}:
\begin{lemma}\label{lm:n4Overlaps} The only partial overlaps among
  cones in the tropical secant graph of a monomial curve in $\pr^4$ are:
  \begin{enumerate}\item 
    $F_{i_1,i_2,i_3}D_{i_2}$ and $D_{i_2}E_{i_2}$, where
    $(i_4-i_2)i_1 = (i_4-i_3)i_2$. In this case, $E_{i_2}$ lies in the
    interior of the edge $F_{i_1,i_2,i_3}D_{i_2}$, and so we
    replace the edge $F_{i_1,i_2,i_3}D_{i_2}$ in the tropical
    secant graph by the two edges $F_{i_1,i_2,i_3}E_{i_2}$ 
    (with weight $m_{F_{i_1,i_2,i_3}D_{i_2}}$), and $D_{i_2}E_{i_2}$ (with new weight
    $m_{D_{i_2}E_{i_2}}+m_{F_{i_1,i_2,i_3}D_{i_2}}$).
  \item $F_{\aaa}D_{i_0}$ and $F_{\ap}D_{i_0}$, where
    $\aaa=\{0,i_l,i_t\}$ and $\ap=\{0,i_u,i_t\}$,
    $l>t>u$. Furthermore, if $m$ denotes the remaining index, then
    $i_m+i_t= i_l+i_u$ and $l>m>u$. Hence, $F_{\aaa} \in
    F_{\ap}D_{i_0}$, and we replace the edge $F_{\ap}D_{i_0}$
    by the edges ${F_{\ap} F_{\aaa}}$ 
(with weight $m_{F_{\ap}D_{i_0}}$), and ${F_{\aaa}D_{i_0}}$
    (endowed with the new weight $m_{F_{\aaa}D_{i_0}} +
    m_{F_{\ap}D_{i_0}}$).
  \item $F_{\aaa}D_{i_j}$ and $F_{\ap}D_{i_j}$, where
    $\aaa=\{0,i_j,i_t\}$ and $\ap=\{i_j,i_t,i_u\}$. Furthermore, if
    $m$ is denotes the remaining index, then $i_u=i_t+i_m$. Assume
    $t<j$. Then, $F_{\aaa} \in F_{\ap}D_{i_j}$. We replace the edge
    $F_{\ap}D_{i_j}$ by the edges $F_{\ap}F_{\aaa}$ (with weight
    $m_{F_{\ap}D_{i_j}}$), and $F_{\aaa}D_{i_j}$ (changing its weight
    to $m_{F_{\aaa}D_{i_j}} + m_{F_{\ap}D_{i_j}}$).  On the contrary,
    if $j<t$, then $F_{\ap} \in {F_{\aaa}D_{i_j}}$, and we replace the
    edge $F_{\aaa}D_{i_j}$ by the edges ${F_{\ap} F_{\aaa}}$ (with
    weight $m_{F_{\aaa}D_{i_j}}$) and ${F_{\ap}D_{i_j}}$ (with the new
    weight $m_{F_{\aaa}D_{i_j}} + m_{F_{\ap}D_{i_j}}$).
\end{enumerate}
\end{lemma}
\begin{proof} Let $\mathbf{l_2}=(0, i_1, i_2, i_3, i_4)$.
For any $\mu\geq 0$, each intersection point can be written as:
  \begin{enumerate}\item 
    $\displaystyle{ \frac{i_2(i_4-i_3)}{i_4}\cdot F_{\aaa} +
      (\frac{i_2(i_3-i_2)}{i_4} + \mu) \cdot D_{i_2} =
 E_{i_2} + \mu \cdot D_{i_2} +
      \frac{-i_2}{i_4} \cdot \mathbf{l_2}}$,
  \item     $\displaystyle{ 
 (i_l-i_m)\cdot F_{\aaa} + \mu \cdot D_{i_0} = (i_m-i_u) \cdot
    F_{\ap} + (i_t+ \mu) \cdot D_{i_0} + (-i_m) \cdot {\bf 1} + \mathbf{l_2}
}$,
\item
        $\displaystyle{ i_m \cdot F_{\aaa} + (i_t+\mu) \cdot D_{i_j} = i_t\cdot F_{\ap} +
    (i_j+\mu) \cdot D_{i_j} + i_m \cdot {\bf 1} -
\mathbf{l_2}}$.
  \end{enumerate}
\vspace{-0.5cm}
\end{proof}
From the previous lemma, we get a modification of the tropical secant
graph possibly with internal crossings but without partial overlaps.
Finally, using this new graph, we construct the Gr\"obner tropical
secant graph as indicated by the following theorem, which we
illustrate in Example~\ref{ex:ranestadrevisited}:
\begin{theorem}\label{thm:GrobnerTropSecGraphForN=4}
  The Gr\"obner tropical secant graph for the monomial curve
  $(1:t^{i_1}:t^{i_2}: t^{i_3}: t^{i_4})$ in $\pr^4$ may be obtained
  by adding finitely many internal crossings to the tropical secant
  graph (after modifications using Lemma~\ref{lm:n4Overlaps}).
  These internal crossing come in two types. The first one consists of
  points $P_{\aaa, j, \ap, k} \in (\cT C + \RR_{\geq 0} \langle
  F_{\aaa},D_{i_j}\rangle )\bigcap (\cT C + \RR_{\geq 0}\langle
  F_{\ap,}D_{i_k}\rangle )$, where:
  \begin{enumerate}
  \item 
$\aaa=\{i_j,i_t\}, \ap=\{i_j,i_t, i_l, i_k\}$, where either $j>t>s>k$ and
    $u>s$ or $j<t<l<k$ and $u<l$;
\item 
$\aaa=\{i_j,i_t\}, \ap=\{i_j, i_l, i_k\}$,
where either $t>u>l>k, j>l$ and $i_j+i_u\geq i_l+i_t$ or
$t<u<l<k, j<s$ and $i_j+i_u\leq i_l+i_t$;
\item 
 $\aaa=\{i_j,i_t\}, \ap=\{i_t, i_l, i_k\}$, 
where either $u>l>k, j>t>l$ and $i_j+i_l\geq i_u+i_t$ or
$u<l<k, j<t<l$ and $i_j+i_l\leq i_u+i_t$;
\item 
$\aaa=\{i_t,i_j,i_u,i_k\}, \ap=\{i_j, i_u, i_k\}$,
where either $j>u>k, t>u>k$ and $t>l$  or
$j<u<k, t<u<k$ and $t<l$;
\item 
$\aaa=\{i_j,i_u,i_k\}, \ap=\{i_u, i_t, i_k\}$, 
where either $j>s>t, j>u>k, u>t$ and $i_j+i_t\geq i_u+i_l$  or
$j<l<t, j<u<k, u<t$ and $i_j+i_t\leq i_u+i_l$;

\item 
 $\aaa=\{i_u,i_j,i_k\}, \ap=\{i_j,i_k, i_t\}$, 
where $u>l>t, j>t, u>k, j>k$, $i_l+i_j\geq i_u+i_t\geq i_l+i_k$; 

\item 
$\aaa=\{i_j,i_u,i_k\}, \ap=\{i_k, i_t\}$, where either $j>u>l>t, u>k$ and
$i_u+i_t\geq i_l+i_k$ or
$j<u<l<t, u<k$ and
$i_u+i_t\leq i_l+i_k$;

\item 
$\aaa =\{i_j, i_l, i_u, i_k\}, \ap=\{i_u, i_t, i_k\}$, where either
$j>l>u>k$ and $u>t$ or $j<l<u<k$ and $u<t$;

\item 
$\aaa=\{i_l,i_j, i_u,i_k\}, \ap=\{i_j,i_u, i_t, i_k\}$, where $s>u>t$ and $j>u>k$;
\item 
$\aaa=\{i_4,i_3, i_2\}, \ap=\{i_2,i_1,0\}$, $j=4$ and $k=0$;

\item 
$\aaa=\{i_4,i_3\}, \ap=\{ i_1, i_0\}$,  $j=4$ and $k=0$.
  \end{enumerate}
The second class satisfies $P_{j,\aaa
    , k} \in (\cT C + \RR_{\geq 0} \langle E_{i_j},E_{i_j+1}\rangle  )\bigcap
(\cT C + \RR_{\geq 0}  \langle F_{\aaa},D_{i_k}\rangle )$, where:
\begin{enumerate}
\item $\aaa=\{i_1, i_2,i_3\}$, $j=k=2$ and $i_1(i_4-i_2)\geq i_2(i_4-i_3)$;
\item $\aaa=\{i_1, i_2,i_4\}$, $j=k=2$ and $i_3(i_2-i_1)\geq i_2(i_4-i_1)$;
\item $\aaa=\{i_1, i_2,i_3\}$, $j=1$, $k=2$ and $i_4(i_2-i_1) \geq i_2(i_3-i_1)$;
\item $\aaa=\{i_0, i_2,i_3\}$, $j=1$, $k=2$ and $i_3(i_2-i_1)\geq i_2(i_4-i_1)$.
\end{enumerate}
\end{theorem}
\begin{proof}
  The proof is very similar to the one we outlined for
  Theorem~\ref{thm:GrobnerTropSecGraphForN=5}, so we only give the
 linear combination expressing each intersection point described in the
  statement. As usual, we let $\mathbf{l_2}=(0,i_1,i_2,i_3,i_4)\in \Lambda
  \subset \ZZ^5$.  The internal crossings of the first type are listed
  below:
  \begin{enumerate}
  \item 
For simplicity, assume $j>t>l>k$ and $u>l$ (if not, we multiply
    the expression by $-1$ to obtain the intersection point):
\[
(i_t-i_l) \cdot F_{i_j, i_t} + (i_j-i_t) \cdot D_{i_j} = (i_u-i_l)
\cdot F_{i_j, i_t, i_l, i_k} + (i_l-i_k) \cdot D_{i_k} -i_u \cdot {\bf 1} +
\mathbf{l_2}.
\]
\item 
Assume $i_j+i_u\geq i_l+i_t$. The intersection point is:
\[
(i_t-i_u) \cdot F_{i_t, i_j} + (i_j+i_u-i_l-i_t) \cdot D_{i_j} =
(i_u-i_l) \cdot F_{i_j, i_l, i_k} + (i_l-i_k) \cdot D_{i_k} - i_u
\cdot {\bf
  1} + \mathbf{l_2}.
\]
Note that if $i_j+i_u= i_l+i_t$, the intersection point is
$F_{i_j,i_t}$.
\item 
Assume $i_j+i_l\geq i_u+i_t$. The intersection point is:
\[
(i_t-i_l) \cdot F_{i_t, i_j} + (i_j+i_l-i_u-i_t) \cdot D_{i_j} =
(i_u-i_l) \cdot F_{i_t, i_l, i_k} + (i_l-i_k) \cdot D_{i_k} - i_u\cdot
{\bf
  1} + \mathbf{l_2}.
\]
Again, if $i_j+i_u= i_l+i_t$, the intersection point is
$F_{i_j,i_t}$.
\item 
Assume $u>k$. The intersection point is:
\[
(i_t-i_l) \cdot F_{i_t, i_j,i_u,i_k} + (i_j-i_u) \cdot D_{i_j} =
(i_t-i_u) \cdot F_{i_j, i_u, i_k} + (i_u-i_k) \cdot D_{i_k} - i_l\cdot
{\bf
  1} + \mathbf{l_2}.
\]
\item 
Assume $i_j+i_t\geq i_u+i_l$. The intersection point is:
\[
(i_u-i_t) \cdot F_{i_j,i_u,i_k} + (i_j+i_t-i_u-i_l) \cdot D_{i_j} =
(i_l-i_t) \cdot F_{i_u, i_t, i_k} + (i_u-i_k) \cdot D_{i_k} - i_l\cdot
{\bf
  1} + \mathbf{l_2}.
\]
If $i_j+i_t=i_u+i_l$, the intersection point is $F_{i_j,i_u,i_k}$.
\item 
By symmetry, we can assume $j>k$. The intersection point is:
\[
(i_r-i_l) \cdot F_{i_r,i_j,i_k} + (i_l+i_j-i_r-i_t) \cdot D_{i_j} =
(i_l-i_t) \cdot F_{i_j, i_k, i_t} + (i_r+i_t-i_l-i_k) \cdot D_{i_k} - i_l\cdot
{\bf
  1} + \mathbf{l_2}.
\]
At most one of the equalities $i_l+i_j=i_i+i_t$, $i_r+i_t=i_l+i_k$
holds, and in this case, $F_{\aaa}$ or $F_{\ap}$ is the
intersection point.
\item 
Assume $i_r+i_t\geq i_l+i_k$. The intersection point is:
\[
(i_r-i_l) \cdot F_{i_r,i_j,i_k} + (i_j-i_r) \cdot D_{i_j} =
(i_l-i_t) \cdot F_{i_k, i_t} + (i_r+i_t-i_l-i_k) \cdot D_{i_k} - i_l\cdot
{\bf
  1} + \mathbf{l_2}.
\]
If $i_r+i_t=i_l+i_k$, the intersection point is $F_{i_k,i_t}$.
\item 
Assume $u>t$. The intersection point is:
\[
(i_u-i_t) \cdot F_{i_j,i_l,i_u,i_k} + (i_j-i_l) \cdot D_{i_j} =
(i_l-i_u) \cdot F_{i_u, i_t,i_k} + (i_u-i_k) \cdot D_{i_k}+ (i_u-i_t-i_l)\cdot
{\bf
  1} + \mathbf{l_2}.
\]
\item 
 The intersection point is:
\[
(i_u-i_t) \cdot F_{i_j,i_l,i_u,i_k} + (i_j-i_u) \cdot D_{i_j} =
(i_l-i_u) \cdot F_{i_j,i_u, i_t,i_k} + (i_u-i_k) \cdot D_{i_k}+ (i_u-i_t-i_l)\cdot
{\bf
  1} + \mathbf{l_2}.
\]
\item 
 The intersection point is:
\[
(i_2-i_1) \cdot F_{i_4,i_3,i_2} + (i_4-i_3) \cdot D_{i_4} =
(i_3-i_2) \cdot F_{i_2, i_1,0} + i_1 \cdot D_{0}+ (i_2-i_3-i_1)\cdot
{\bf
  1} + \mathbf{l_2}.
\]
\item 
The intersection point is:
\[
(i_3-i_2) \cdot F_{i_4, i_3} + (i_4-i_3) \cdot D_{i_4} =
(i_2-i_1) \cdot F_{i_1, i_0} + i_1 \cdot D_{i_0} - i_2 \cdot {\bf
  1} + \mathbf{l_2}.
\]
\end{enumerate}
We list the internal crossings of the second type:
\begin{enumerate}
\item
The intersection point is:
\[
\frac{i_3-i_1}{i_3-i_2}\cdot E_{i_2} +
\frac{i_1(i_4-i_2)-i_2(i_4-i_3)}{(i_4-i_3)(i_3-i_2)}\cdot E_{i_3} =
i_1  \cdot F_{i_1, i_2, i_3} + (i_2-i_1)  \cdot D_{i_2} 
+ \frac{i_1}{i_4-i_3} \cdot  \mathbf{l_2}.
\]
The positivity of the coefficient of $E_{i_3}$ gives the inequality
constraint in the statement. Note that if $i_1(i_4-i_2)=i_2(i_4-i_3)$
then the crossing point is $E_{i_2}$, which is already a node in the
graph, but in this case it lies in the interior of the edge
$F_{i_1,i_2,i_3}D_{i_2}$. 
\item
The internal crossing is:
\[
\frac{i_3}{i_3-i_2}\cdot E_{i_2}
+\frac{i_3(i_2-i_1)-i_2(i_4-i_1)}{(i_4-i_3)(i_3-i_2)} \cdot E_{i_3} =
i_1 \cdot F_{i_1, i_2, i_4} + i_2 \cdot D_{i_2} -\frac{i_1}{i_4-i_3}
\cdot \mathbf{l_2},
\]
with the positivity constraint for the coefficient of $E_{i_3}$. If
$i_3(i_2-i_1)=i_2(i_4-i_1)$ the crossing point is $E_{i_2}$ and it
lies in the interior of the edge $F_{i_1,i_2,i_4}D_{i_2}$. 
\item The intersection point is:
\[
\frac{i_4(i_2-i_1)-i_2(i_3-i_1)}{i_1(i_2-i_1)}\cdot E_{i_1} + \frac{i_3-i_1}{i_2-i_1}\cdot E_{i_2} =  (i_4-i_3)\cdot F_{i_1, i_2, i_3} + (i_3-i_2) \cdot D_{i_2} +    \mathbf{l_2},
\]
with the positivity constraint for the coefficient of $E_{i_1}$.  If
$i_4(i_2-i_1)=i_2(i_3-i_1)$ the crossing point is $E_{i_2}$ and it
lies in the interior of the edge $F_{i_1,i_2,i_3}D_{i_2}$.
\item The intersection point is:
\[
\frac{i_3(i_2-i_1)-i_2(i_4-i_1)}{i_1(i_2-i_1)}\cdot E_{i_1} + \frac{i_4-i_1}{i_2-i_1}\cdot E_{i_2} = (i_4-i_3) \cdot F_{i_0, i_2, i_3} +  (i_3-i_2)\cdot D_{i_2} 
-  (i_4-i_3)\cdot {\bf
  1} + \mathbf{l_2},
\]
with positivity constraint for the coefficient of $E_{i_1}$.  If
$i_3(i_2-i_1)=i_2(i_4-i_1)$ the crossing point is $E_{i_2}$ and it
lies in the interior of the edge $F_{i_0,i_2,i_3}D_{i_2}$.
  \end{enumerate}

  \vspace{-0.55cm}\end{proof} 
\begin{example} [\textbf{Example~\ref{ex:Ranestad}, revisited}]\label{ex:ranestadrevisited}
  As we saw, there are no new overlaps of edges, so
  Lemma~\ref{lm:n4Overlaps} does not apply here. Using
  Theorem~\ref{thm:GrobnerTropSecGraphForN=4}, we can explain the six
  new small green nodes we added to build the Gr\"obner tropical
  secant graph (Figure~\ref{fig:RanestadSecantAndGrobnerSecant}). They
  come from six internal crossings of the first type between the edges
  $F_{\aaa}D_{i_j}$ and $F_{\ap}D_{0}$, where: $ \aaa=\{55,78\}, j=
  78, \ap=\{0,30, 78\}$ (case (ii)
  ); $\aaa=\{55,78\}, j= 78, \ap=\{0,30\}$ (case (xi)
  ); $\aaa=\{45,78\}, j= 78, \ap=\{0,30, 45, 78\}$ (case (i)
  ); $ \aaa=\{0,30,45,78\}, j=45, \ap=\{0,30, 45\}$ (case (iv)
  ); $ \aaa=\{45,78\}, j=78, \ap=\{0,30, 45\}$ (case (iii)
  ); and $ \aaa=\{45,55\}, j= 55, \ap=\{0,30, 45, 55\}$ (case (i)
  ).
\end{example}


\section{Chow polytopes, tropical secants of lines, toric arrangements and beyond}
\label{sec:chow-polytope-secant}

The implicitization methods discussed in the previous section can be
generalized to monomial curves in higher dimensional projective
spaces. In this
case, one can recover the \emph{Chow polytope} of the secant variety
by a generalization of the ray-shooting algorithm, known as the
\emph{orthant-shooting} algorithm \citep[Theorem
2.2]{TropDiscr}. Instead of shooting rays, we shoot orthants (i.e.\
cones spanned by subsets of the canonical basis of $\RR^{n+1}$) of
dimension equal to $n-3$ (the codimension of our variety). A formula
similar to the one described in Theorem~\ref{thm:RayShoot} gives us
the vertex of the Chow polytope associated to a sufficiently generic
input objective vector. However, it is not easy to given an analog
to the walking algorithm. The difficulty comes from the fact that, a
priori, there is no canonical way of walking along the complement of
the tropical variety. Recently, Alex Fink has developed a method to
reduce the computation of the Chow polytope to the codimension one
setting, based on the 
{orthant shooting algorithm}~\cite{ChowTropical}. His approach
allows us to use the techniques discussed for the secant hypersurface
case~\cite{ChowTropical}. Thanks to his results, existing software
from~\cite{Mega09} can be used in higher codimension examples, such
as 
rational normal curves in $\pr^n$.

Before giving a numerical example, we explain the method presented
in~\cite{ChowTropical} for computing Chow polytopes. We define a map,
called the \emph{Chow map $ch$}, which takes a tropical variety $\cT X$ of
dimension $d$ in $\TP\pr^n$ to its \emph{tropical Chow hypersurface},
$ch(\cT X)=\cT{X}\boxplus {\mathcal{L}_{n-d-1}}^{\refl}$
\cite[Definition 5.2]{ChowTropical}. The set $ch(\cT X)$ is 
precisely the tropicalization of a hypersurface in $\pr^n$ whose
Newton polytope equals the desired Chow polytope of $X$~\cite[Theorem
5.1]{ChowTropical}. 

We now describe the tropical hypersurface $ch(\cT X)$ as the union of
weighted $(n-1)$-dimensional cones in $\TP\pr^n$. First, we pick all
weighted maximal $d$-dimensional cones $(\sigma, m_{\sigma})$ in $\cT
X\subset \TP\pr^n$, and all $n-d-1$ cones $C_{J}$ generated by subsets
of $n-d-1$ vectors among $\{-e_0, -e_1, \ldots, -e_n\}$, i.e.\ the
negative of the elements in the canonical basis of $\RR^{n+1}$. The
subscript $J$ indicates the indices of the vectors chosen from
this basis. For simplicity, we assume that each cone $\sigma$ is
simplicial and that it is spanned by integer vectors $\{v^{\sigma}_1,
\ldots, v^{\sigma}_d\}$ in $\TP\pr^n$. Secondly, we take the Minkowski sum
of $\sigma$ and $C_J$ for every possible pair, and we check if the
cone $\sigma +C_J$ in $\TP\pr^n$ has codimension 1. If so, this means
that the matrix
\[
A:=\left(\begin{array}[c]{c|c|c|c|c|c}
  v^{\sigma}_1 & \ldots & v^{\sigma}_d & -e_{j_1} & \ldots & -e_{j_{n-d-1}}
\end{array}
\right),
\]
is full dimensional, for $J=\{j_1, \ldots, j_{n-d-1}\}$. We weight
the  new cone $\sigma + C_J$ by 
\begin{equation}
m_{\sigma+C_J}:=m_{\sigma}\cdot \gcd(\text{maximal } (n-d-1)\times (n-d-1)\text{--minors of
}A)\label{eq:1}.
\end{equation}
If the matrix $A$ is not of full rank, we discard the cone $\sigma
+C_J$ from the list of valid combinations and we move on to the next
pair. The set $ch(\cT X)$ is the union of the valid combinations, with
weights given by formula~(\ref{eq:1}).

\begin{example}\label{ex:RNC}
  The canonical example of a monomial curve in $\pr^{n}$ is the \emph{rational
  normal curve} $t\mapsto (1:t:t^2:\ldots:t^n)$. 
These curves and their secants have been extensively studied in the
past. They are known to be determinantal varieties (\cite[Proposition
9.7]{Harris},\cite[Proposition 2.2]{HankelSecants}) defined by the
$3\times 3$-minors of the $j\times (n-j+2)$ \emph{Hankel} matrix:
\[
x_j:=
\left(
  \begin{array}{ccccc}
x_1 & x_2 & x_3 & \ldots & x_{n-j+2}\\
x_2 & x_3 & \ldots & \ldots & \ldots\\
x_3 &  \ldots & \ldots & \ldots & \ldots\\
\vdots  & \vdots & \vdots &\vdots & \vdots\\
x_j & x_{j+1} & \ldots & x_{n} & x_{n+1}
\end{array}
\right).
\]
The ideal generated by the $3\times 3$-minors of this matrix is
independent of the index $j$ \citep[Corollary 2.2]{Conca} and it is a
set-theoretic complete intersection \cite{Valla} of degree
$\binom{n-1}{2}$ \cite{Conca}.

Using Theorem~\ref{thm:MainThm}, we compute the tropical secant
graph of the rational normal curve in $\pr^n$.  It has
 $n+1$ nodes $D_{j}=e_{j}$ ($0\leq j \leq n$), $n-3$ nodes
$E_{j}=(0, 1, 2, \ldots, j , \ldots, j)$ ($2\leq j \leq n-2$) and
$(\lfloor n/2\rfloor + 1) (\lfloor n/2\rfloor -2)/2$ nodes $F_{\aaa}$,
where $\aaa = 
\aaa_{r,u}:=
\{ u+ k \cdot r: k\in \NN \}
\cap \{0, \ldots, n\}$ for some $0\leq u < r$ and $1<r< \lfloor
n/2\rfloor$. In addition, it has one or two nodes $F_{\aaa}$ where
$r=\lfloor n/2\rfloor$ and $u=0$ (for $n$ even), or $u=0,1$ (for $n$
odd).

The graph has $2n-5$ edges labeled $E_j E_{j+1}$, ($2\leq j \leq n-3$), $E_j
D_j$ ($2\leq j \leq n-2$), $D_1E_2$, $E_{n-2}D_{n-1}$,  with weight
one. It also has edges $F_{\aaa_{r,u}}D_{j}$ ($j\in \aaa_{r,u}$),
 with weight $\varphi(r)/2$ if $r>2$, or weight 1 if
$r=2$. In addition, it has  $\lceil n/2\rceil (n+3+\lfloor
n/2\rfloor)/2$ edges $D_jD_{j+r}$ ($0\leq j <n-r$ and $r> \lfloor
n/2\rfloor$) with weight $\varphi(r)/2$, and $\lfloor n/2\rfloor -2$ 
edges $D_{j}D_{j+\lfloor n/2\rfloor}$ ($2\leq j \leq \lfloor
n/2\rfloor$) with weight $\varphi( \lfloor n/2\rfloor)/2$. Finally, 
if $n$ is even, the edge $D_{1} D_{1+n/2}$ has weight
$\varphi(n/2)/2$ if $n>2$ and weight $1$ if $n=2$.

We illustrate the previous construction in the case $n=4$.  After
removing the bivalent node $D_2=e_2$, 
we are left with a graph with six nodes $D_0=e_0$,
$D_1=e_1$, $D_3=e_3$, $D_4=e_4$, $E_2=(0,1,2,2,2)$ and
$F_{0,2,4}=(1,0,1,0,1)$ and nine edges, all with trivial weight
$1$. This graph is the 1-skeleton of a bipyramid
(Figure~\ref{fig:RNC4,6}).
\begin{figure}
  \centering
 \begin{minipage}[c]{.41\linewidth}
  \includegraphics[scale=0.3]{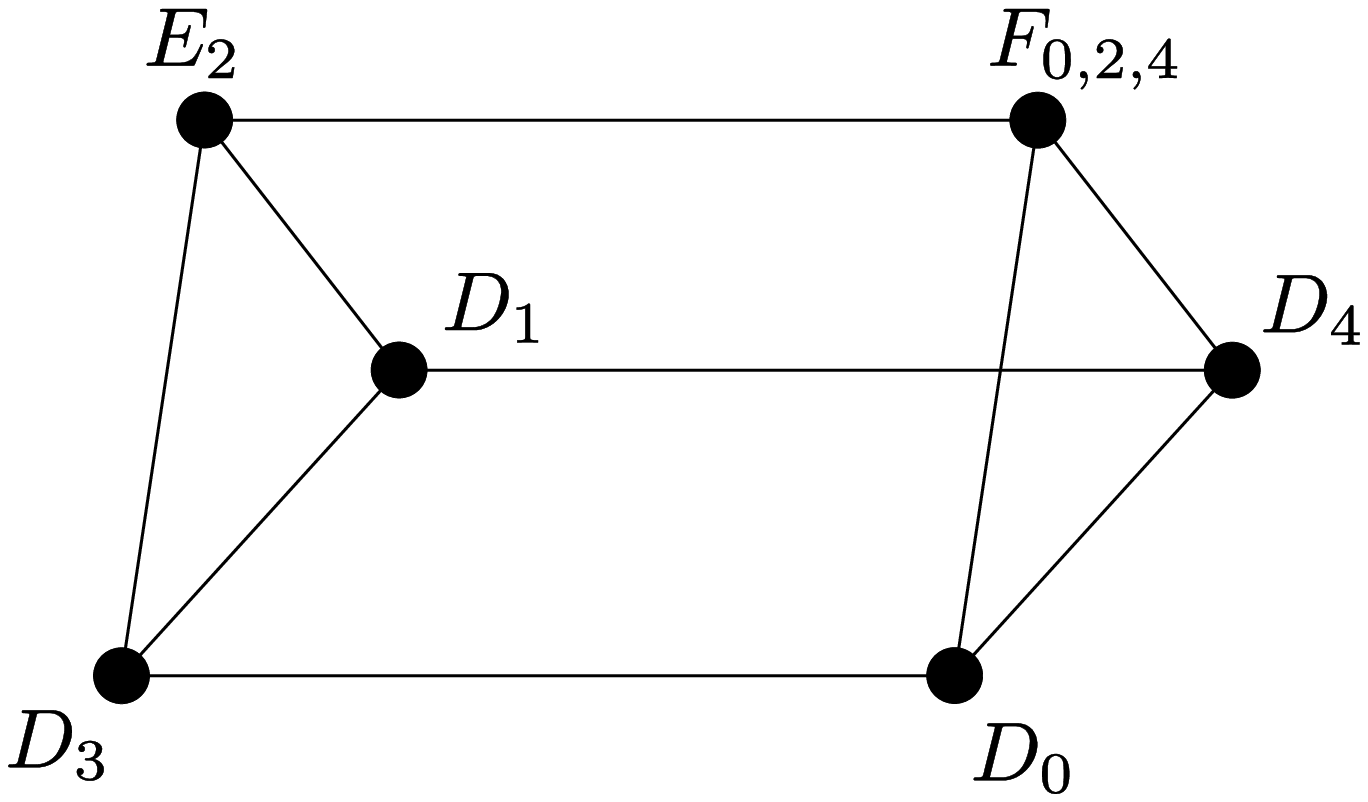}
\end{minipage}
 \begin{minipage}[c]{.5\linewidth}
\includegraphics[scale=0.2]{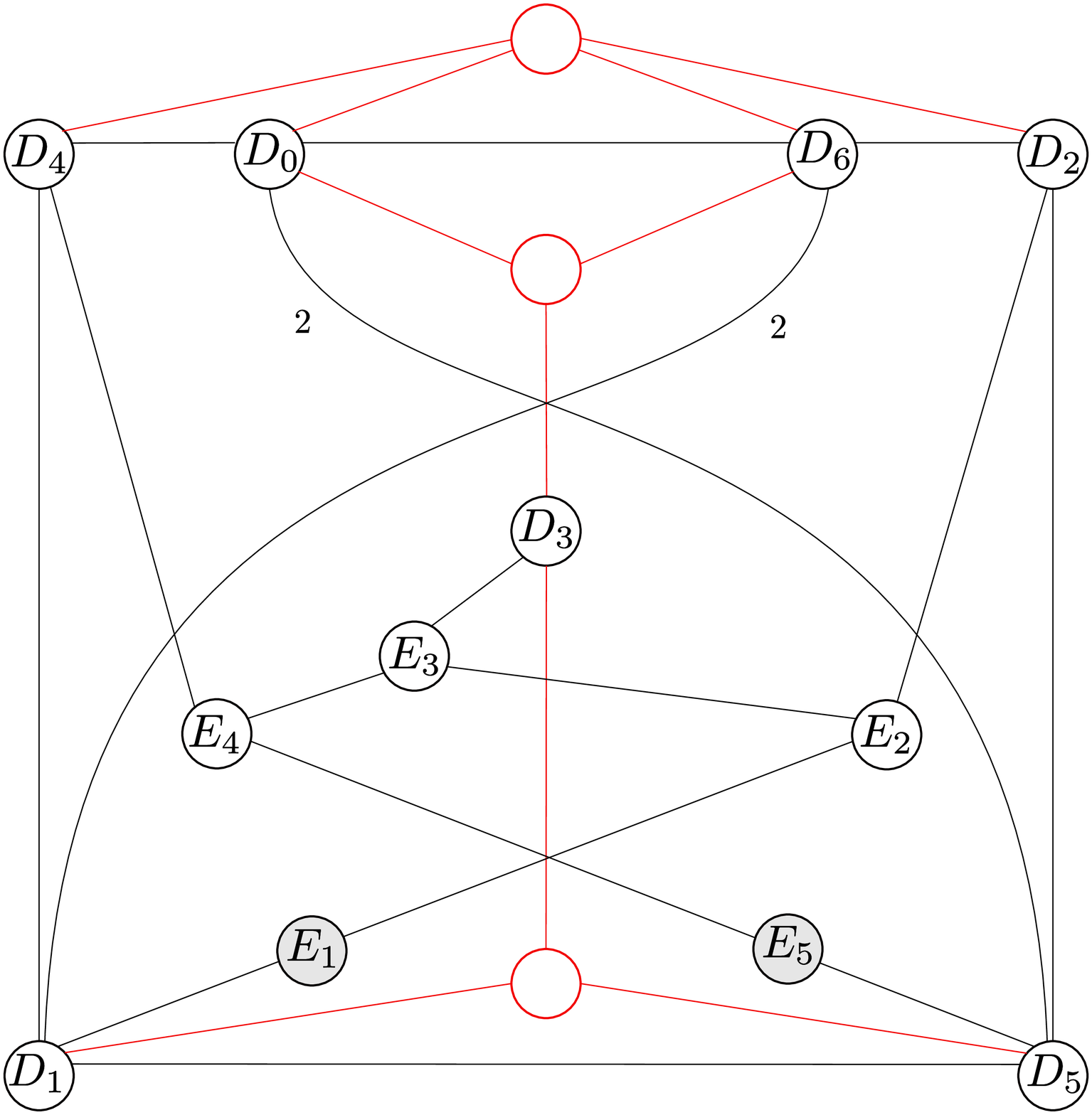}
\end{minipage}
  \caption{The Gr\"obner tropical secant graphs of the rational normal curves in
    $\pr^4$ and $\pr^6$.}
  \label{fig:RNC4,6}
\end{figure}

Using ray-shooting and walking algorithms, we see that the equation
has multidegree $(3,6)$ with respect to the lattice
$\Lambda=\ZZ\langle {\bf 1}, (0,1,2,3,4)\rangle$ and its Newton
polytope has vertices $(0, 0, 3, 0, 0)$, $(0, 1, 1, 1, 0)$, $(1, 0, 0,
2, 0)$, $(0, 2, 0, 0, 1)$, $(1, 0, 1, 0, 1)$, and $f$-vector
$(5,9,6)$. This graph is precisely the tropical discriminant of the
Veronese surface, regarded as the projectivization of the variety of
all symmetric $3\times 3$ matrices of rank at most one \citep[Example
4.4]{TropDiscr}. Its defining equation is a dehomogeneization of the
Hankel $3\times 3$-determinant.

The case $n=6$ was computed in \cite[Example 20]{BJSST}. Its defining
ideal is generated by the $3\times 3$-minors of the $4\times 4$-Hankel
matrix. After changing the signs of the rays obtained by \gfan\ (to
agree with our \emph{min} convention) and reducing modulo the lattice
$\Lambda$, we see that our construction matches theirs for all rays
except for the ones corresponding to three bivalent node $E_1, E_5$
and one that is absent in our graph from Figure~\ref{fig:RNC4,6} (nodes 2, 3 and 15 in the notation
of~\cite{BJSST}).  We follow our convention from
Example~\ref{ex:Ranestad} and keep the red nodes $F_{\aaa}$
unlabeled. All weights in our graph equal 1, except for two edges with
weight 2 (indicated in Figure~\ref{fig:RNC4,6}). From this
computation, we confirm that the Gr\"obner tropical secant graph
coincides with the tropical secant graph for $n\geq 6$.

We now use the methods from \cite{ChowTropical} for computing
the Chow polytope of the secant variety of the rational normal curve
in $\pr^6$. In this case, the polytope has 289 vertices and $f$-vector
$(289, 897, 981, 442, 71)$. All vertices have multidegree
$(30,90)$, which matches the formula $\deg ch(X)=\codim X \cdot \deg
X$ from \cite[Lemma 3.4]{ChowTropical}, since $30 = 3 \cdot
\binom{5}{2}$. 

Changing the grading to reflect the torus action by the exponent
vector $(0,1,2,\ldots, n)$ rather than the action by the all-ones
vector, gives us the weighted projective space
$\pr^{n}_{(0,1,2,\ldots, n)}$ as the ambient space, rather than the
usual $\pr^n$. We conjecture that this new setting results in the
formula ${\deg}'\, ch(Sec^1(C)) = (n-3) \cdot {\deg}'\, Sec^1(C)$
connecting the degree of $ch(Sec^1(C))$ with respect to this
exponent vector to the codimension and the degree of the secant
variety of $C$ inside this new toric
variety. 
\end{example}

We now switch gears to study the set of all tropical lines between
points in the tropicalization of a monomial projective curve. We aim
to highlight the differences between this set and the tropicalization
of the first secant variety of the same curve.  By definition, a
tropical line segment between two points in the tropical curve $\cT C$
is the loci of all points obtained as the coordinatewise minima (i.e.\
tropical addition) of two fixed points in the  classical
plane spanned
by the lattice $\Lambda=\langle {\bf 1}, (0, i_1, \ldots,
i_n)\rangle$. We interpret this as the line spanned by the vector
$(0,i_1, \ldots, i_n)$ in
$\TP\pr^n$. 
The set of all tropical lines between points in $\cT C$ is often
denoted by $S^1(\cT C)$ and it is called the \emph{first tropical
  secant variety} of the line $\cT C\subset \TP \pr^n$. Since $S^1(\cT
C)$ is the image of the tropicalization of the secant map $\phi$
from~(\ref{eq:SecantMap}), we know it is contained in the
tropicalization of the image of $\phi$, hence $S^1(\cT C)$ is
contained in $\cT Sec^1(C)$.  These two tropical sets have been
compared and their rich combinatorial structures has been studied by many,
including Develin and Draisma~\cite{DevelinTropSecant,DraismaTropSecant}. In particular,
by~\cite[Corollary 2.2]{DevelinTropSecant}, we know that $S^1(\cT C)$
is a cone from $\cT C$ over a polytopal complex, called the
\emph{first tropical secant complex} of $\cT C$. We aim to describe
this complex for the case of monomial curves.

Each point in $S^1(\cT C)\subset \RR^{n+1}$ may be thought of as a
height vector for a configuration of points $\{0, i_1, \ldots, i_n\}$
on $\RR$ which induces a regular subdivision of the convex hull
defined by these $n+1$ points.  The faces of this polytopal complex
correspond to regular subdivisions such that two facets cover all
points. These faces are ordered by refinements of the corresponding
subdivisions.  Since by assumption, our exponent vector has distinct
coordinates, the classical line $\cT C$ is \emph{generic} in the sense
of Develin, and~\cite[Theorem 3.1]{DevelinTropSecant} gives a very
nice characterization of this complex. It is precisely the set of
lower faces of the cyclic polytope $C(2,n-1)$, defined as the convex
hull of $n-1$ generic points in the parabola $\{y=x^2\}\subset \RR^2$.
It follows immediately that this complex is a chain graph with $n-1$
vertices. Figure~\ref{fig:FirstTropicalSecantComplexInP4} illustrates
this construction for a generic classical line in $\TP\pr^4$.
\begin{figure}[htb]
  \centering
  \includegraphics[scale=0.3]{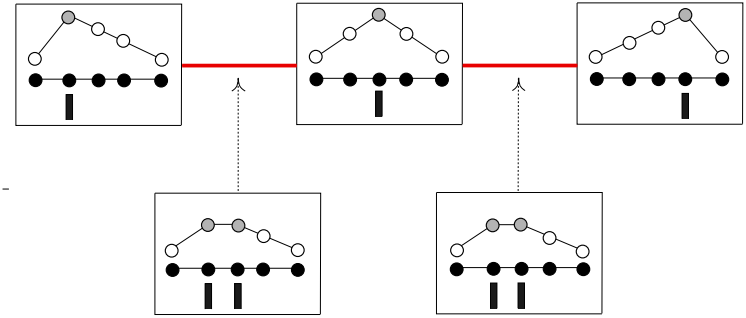} 
  \caption{The first tropical secant complex of the line $\RR\langle
    (0, i_1, i_2, i_3, i_4)\rangle$ in~$\TP\pr^4$.}
\label{fig:FirstTropicalSecantComplexInP4}
\end{figure}

\begin{proposition}
\label{pr:TropSecantComplex}
  The first tropical secant complex of the tropicalization of the monomial
  curve $(1:t^{i_1}:\ldots:t^{i_n})$ (with $0<i_1<\ldots<i_n$) is a chain graph with $n-1$
  vertices $v^{(1)}, \ldots, v^{(n-1)}$ embedded in $\RR^{n+1}$ by 
\begin{equation*}
v^{(k)}=\sum_{j\leq k} i_j\, e_j + \sum_{k<j<n}
\frac{i_k}{i_n-i_k}(i_n-i_j) \,e_j \;\qquad k=1, \ldots, n-1.
\end{equation*}
\end{proposition}
Note that $v^{(k)}$ corresponds to the regular subdivision of the
configuration  $\{0, i_1, \ldots, i_n\}$ with exactly two
facets $\{0, i_1, \ldots, i_k\}$ and $\{i_{k}, \ldots, i_n\}$. It is
embedded as a height vector, where the points 0 and $i_n$ have height
zero, the point $i_k$ has height $i_k$ and the remaining points lie in
the interior of these two facets.

\begin{example}
  We describe the first secant complex of the curve
  $(1:t^{30}:t^{45}:t^{55}:t^{78})$, as shown in
  Figure~\ref{fig:FirstTropicalSecantComplexInP4}. It consists of
  three nodes $v^{(1)}= (0, 30, \frac{165}{8}, \frac{115}{8}, 0)$,
  $v^{(2)}=(0, 30, 45, \frac{345}{11},0)$ and $v^{(3)}=(0, 30, 45,
  55,0)$ and two edges $v^{(1)}v^{(2)}$ and $v^{(2)}v^{(3)}$. By
  taking the cone from the linear space $\cT C$ over this complex,  we
  get the first tropical secant variety of the line $\RR\langle (0,
  30, 45, 55, 78)\rangle$.
\end{example}
We now describe $S^1(\cT C)$ as a subgraph of the tropical secant
graph, and we show that the containment of $S^1(\cT C)$ in $\cT
Sec^1(C)$ for a monomial curve $C$ is strict in
general. 
\begin{proposition}\label{pr:tropSecantVsTropicalizedSecant} In the
  notation of Definition~\ref{def:tropSecantGraph}, the first tropical
  secant complex of the tropicalization of the monomial curve
  $(1:t^{i_1}:\ldots:t^{i_n})$ is the \emph{chain subgraph} of the
  tropical secant graph 
  with nodes $E_{i_1}, \ldots,
  E_{i_{n-1}}$.
\end{proposition}
\begin{proof}
  In the notation of Proposition~\ref{pr:TropSecantComplex}, 
 $v^{(k)}$ and $E_{i_k}$ generate the same ray in $\RR\otimes
  \Lambda$ because
\[
v^{(k)} = \frac{-i_k}{i_n-i_k}\cdot (0, i_1, \ldots, i_n) +
\frac{i_n}{i_n-i_k} \cdot E_{i_k}\qquad \qquad \text{for}\;  k=1, \ldots, n-1.
\] The result follows immediately.
\end{proof}

We finish this section by discussing briefly the relationship between
our tropical secant surface graph and compactifications of toric
arrangements. As we saw in Section~\ref{sec:comb-monom-curv}, the
geometric tropicalization of surfaces involves finding a suitable
compactification (a \emph{tropical compactification}) of a parametric
surface inside the torus $\TP^{n+1}$, or, equivalently, of the
complement of $n+1$ divisors in the torus $\TP^2$.  In our setting,
these divisors were the $n+1$ binomial curves $\{w^{i_j}-\lambda=0\}$
($0\leq j \leq n$).  It is well-known that this toric arrangement can
be compactified by a wonderful model in the sense of De
Concini-Procesi~\cite{Wonderful}. Recently, L.\ Moci~\cite{MociPaper}
gave an explicit compactification for toric arrangements and, at first
glance, his techniques are very similar to the ones we discussed in
Section~\ref{sec:comb-monom-curv}. We hope to clarify this connection
more explicit in the near future.


\section*{Acknowledgments}
\label{sec:acknowledgements}

We wish to thank Bernd Sturmfels for suggesting this project to us.
We also thank Melody Chan, Alex Fink and Eric Katz for fruitful
conversations. Special thanks go to Jenia Tevelev for discussions with
the first author on geometric tropicalization, which led to
Theorem~\ref{thm:GT}.


\bibliographystyle{abbrvnat}
\bibliography{bibliography}
\label{sec:biblio}

\bigskip

\end{document}